\pdfoutput=1
\documentclass[oneside,english]{amsart}
\usepackage[T1]{fontenc}
\usepackage[latin9]{inputenc}
\usepackage{verbatim}
\usepackage[norelsize]{algorithm2e}
\usepackage{amsthm}
\usepackage{amstext}
\usepackage{amssymb}
\usepackage{wasysym}
\usepackage{graphicx}

\makeatletter

\providecommand{\tabularnewline}{\\}

\numberwithin{equation}{section}
\numberwithin{figure}{section}
\theoremstyle{plain}
\newtheorem{thm}{\protect\theoremname}
  \theoremstyle{plain}
  \newtheorem{lem}[thm]{\protect\lemmaname}
  \theoremstyle{remark}
  \newtheorem{rem}[thm]{\protect\remarkname}
  \theoremstyle{definition}
  \newtheorem{defn}[thm]{\protect\definitionname}

\usepackage{babel}

\usepackage{babel}
\providecommand{\definitionname}{Definition}
  \providecommand{\lemmaname}{Lemma}
\providecommand{\theoremname}{Theorem}

\usepackage{babel}
\providecommand{\definitionname}{Definition}
  \providecommand{\lemmaname}{Lemma}
\providecommand{\theoremname}{Theorem}

\makeatother

\usepackage{babel}
  \providecommand{\definitionname}{Definition}
  \providecommand{\lemmaname}{Lemma}
  \providecommand{\remarkname}{Remark}
\providecommand{\theoremname}{Theorem}

\begin{document}

\title[Randomized ALS for Canonical Tensor Decomposition]{Randomized Alternating Least Squares for Canonical Tensor Decompositions:
Application to A PDE With Random Data}

\author{Matthew J Reynolds$\,^{*}$, Alireza Doostan$\,^{*}$ and Gregory
Beylkin$\,^{\dagger}$}

\address{$\,^{*}$ Department of Aerospace Engineering Sciences\\
 429 UCB\\
 University of Colorado at Boulder\\
 Boulder, CO 80309\\
 $\,^{\dagger}$Department of Applied Mathematics\\
 UCB 526\\
 University of Colorado at Boulder\\
 Boulder, CO 80309}

\thanks{This material is based upon work supported by the U.S. Department
of Energy Office of Science, Office of Advanced Scientific Computing
Research, under Award Number DE-SC0006402, and NSF grants DMS-1228359, 
DMS-1320919, and CMMI-1454601.}

\keywords{Separated representations; Tensor Decomposition; Randomized Projection;
Alternating Least Squares; Canonical Tensors; Stochastic PDE}
\begin{abstract}
This paper introduces a randomized variation of the alternating least
squares (ALS) algorithm for rank reduction of canonical tensor formats.
The aim is to address the potential numerical ill-conditioning of
least squares matrices at each ALS iteration. The proposed algorithm,
dubbed \emph{randomized ALS}, mitigates large condition numbers via
projections onto random tensors, a technique inspired by well-established
randomized projection methods for solving overdetermined least squares
problems in a matrix setting. A probabilistic bound on the condition
numbers of the randomized ALS matrices is provided, demonstrating
reductions relative to their standard counterparts. Additionally,
results are provided that guarantee comparable accuracy of the randomized
ALS solution at each iteration. The performance of the randomized
algorithm is studied with three examples, including manufactured tensors
and an elliptic PDE with random inputs. In particular, for the latter,
tests illustrate not only improvements in condition numbers, but also
improved accuracy of the iterative solver for the PDE solution represented
in a canonical tensor format.
\end{abstract}

\maketitle

\section{Introduction}

The approximation of multivariate functions is an essential tool for
numerous applications including computational chemistry \cite{APP-DAV:1981,KOL-BAD:2009},
data mining \cite{BA-BE-BR:2008,C-B-K-A:2007,KOL-BAD:2009}, and recently
uncertainty quantification \cite{DOO-IAC:2009,KHO-SCH:2011}. For
seemingly reasonable numbers of variables $d$, e.g. $\mathcal{O}(10)$,
reconstructing a function, for instance, using its sampled values,
requires computational costs that may be prohibitive. This is related
to the so-called ``curse of dimensionality.'' To mitigate this phenomenon,
we require such functions to have special \emph{structures} that can
be exploited by carefully crafted algorithms. One such structure is
that the function of interest $u(z_{1},z_{2},\dots,z_{d})$, depending
on variables $z_{1},z_{2},\dots,z_{d}$, admits a separated representation,
\cite{BEY-MOH:2002,BEY-MOH:2005,KOL-BAD:2009}, of the form

\begin{equation}
u\left(z_{1},z_{2},\dots z_{d}\right)=\sum_{l=1}^{r}\sigma_{l}u_{1}^{l}\left(z_{1}\right)u_{2}^{l}\left(z_{2}\right)\cdots u_{d}^{l}\left(z_{d}\right).\label{eq: separated representation}
\end{equation}
The number of terms, $r$, is called the separation rank of $u$ and
is assumed to be \emph{small}. Any discretization of the univariate
functions $u_{j}^{l}\left(z_{j}\right)$ in (\ref{eq: separated representation})
with $u_{i_{j}}^{l}=u_{j}^{l}\left(z_{i_{j}}\right)$, $i_{j}=1,\dots,M_{j}$
and $j=1,\dots,d$, leads to a Canonical Tensor Decomposition, or
CTD,

\begin{equation}
\mathbf{U}=U\left(i_{1}\dots i_{d}\right)=\sum_{l=1}^{r}\sigma_{l}u_{i_{1}}^{l}u_{i_{2}}^{l}\cdots u_{i_{d}}^{l}.\label{eq:introCTDelements}
\end{equation}

The functions $u_{j}^{l}\left(z_{j}\right)$ in (\ref{eq: separated representation})
and the corresponding vectors $u_{i_{j}}^{l}$ in (\ref{eq:introCTDelements})
are normalized to unit norm so that the magnitude of the terms is
carried by their positive $s$-values, $\sigma_{l}$. It is well understood
that when the separation rank $r$ is independent of $d$, the computation
costs and storage requirements of standard algebraic operations in
separated representations scale linearly in $d$, \cite{BEY-MOH:2005}.
For this reason, such representations are widely used for approximating
high-dimensional functions. To keep the computation of CTDs manageable,
it is crucial to maintain as small as possible separation rank. Common
operations involving CTDs, e.g. summations, lead to CTDs with separation
ranks that may be larger than necessary. Therefore, a standard practice
is to reduce the separation rank of a given CTD without sacrificing
much accuracy, for which the workhorse algorithm is Alternating Least
Squares (ALS) (see e.g., \cite{BEY-MOH:2002,BEY-MOH:2005,BRO:1997,CAR-CHA:1970,HARSHM:1970,KOL-BAD:2009,TOM-BRO:2006}).
This algorithm optimizes the separated representation (in Frobenius
norm) one direction at a time by solving least squares problems for
each direction. The linear systems for each direction are obtained
as normal equations by contracting over all tensor indices, $i=1,\dots,d$,
except those in the direction of optimization $k$.

It is well known that forming normal equations increases the condition
number of the least squares problem, see e.g. \cite{GOL-LOA:1996}.
In this paper we investigate the behavior of the condition numbers
of linear systems that arise in the ALS algorithm, and propose an
alternative formulation in order to avoid potential ill-conditioning.
As we shall see later, the normal equations in the ALS algorithm are
formed via the Hadamard (entry-wise) product of matrices for individual
directions. We show that in order for the resulting matrix to be ill-conditioned,
the matrices for all directions have to be ill-conditioned and obtain
estimates of these condition numbers. To improve the conditioning
of the linear systems, we propose a randomized version of ALS, called\emph{
randomized ALS}, where instead of contracting a tensor with itself
(in all directions but one), we contract it with a tensor composed
of random entries. We show that this random projection improves the
conditioning of the linear systems. However, its straightforward use
does not insure monotonicity in error reduction, unlike in standard
ALS. In order to restore monotonicity, we simply accept only random
projections that do not increase the error.

Our interest here in using CTDs stems from the efficiency of such
representations in tackling the issue of the curse of dimensionality
arising from the solution of PDEs with random data, as studied in
the context of Uncertainty Quantification (UQ). In the probabilistic framework, uncertainties are represented via
a finite number of random variables $z_{j}$ specified using, for
example, available experimental data or expert opinion. An important
task is to then quantify the dependence of quantities of interest
$u(z_{1},\dots,z_{d})$ on these random inputs. For this purpose,
approximation techniques based on separated representations have been
recently studied in \cite{DOO-IAC:2007,NOUY:2007,NOUY:2008,DOO-IAC:2009,NOUY:2010,DO-VA-IA:2013,HACKBU:2012,KHO-SCH:2011,Chinesta11,GR-KR-TO:2013,Hadigol14}
and proven effective in reducing the issue of curse of dimensionality.

The paper is organized as follows. In Section~\ref{sec:Notation-and-Background},
we introduce our notation and provide background information on tensors,
the standard ALS algorithm, and the random matrix theory used in this
paper. In Section~\ref{sec:Randomized-ALS-algorithm}, we introduce
randomized ALS and provide analysis of the algorithm's convergence
and the conditioning of matrices used. Section~\ref{sec:Examples}
contains demonstrations of randomized ALS and comparisons with standard
ALS on three examples. The most important of these examples provides
background on uncertainty quantification and demonstrates the application
of randomized ALS-based reduction as a step in finding the fixed point
solution of a stochastic PDE. We conclude with a discussion on our
new algorithm and future work in Section~\ref{sec:Discussion-and-conclusions}.

\section{Notation and Background\label{sec:Notation-and-Background}}

\subsection{Notation\label{sub:Notation}}

Our notation for tensors, i.e. $d$-directional arrays of numbers,
is boldfaced uppercase letters, e.g. $\mathbf{F}\in\mathbb{R}^{M_{1}\times\dots\times M_{d}}$.
These tensors are assumed to be in the CTD format, 
\[
\mathbf{F}=\sum_{l=1}^{r_{\mathbf{F}}}s_{l}^{\mathbf{F}}\mathbf{F}_{1}^{l}\circ\dots\circ\mathbf{F}_{d}^{l},
\]
where the factors $\mathbf{F}_{i}^{l}\in\mathbb{R}^{M_{k}}$ are vectors
with a subscript denoting the directional index and a superscript
the rank index, and $\circ$ denotes the standard vector outer product.
We write operators in dimension $d$ as $\mathbb{A}=A\left(j_{1},j_{1}^{\prime};\dots;j_{d},j_{d}^{\prime}\right)$,
while for standard matrices we use uppercase letters, e.g. $A\in\mathbb{R}^{N\times M}$.
Vectors are represented using boldfaced lowercase letters, e.g. $\mathbf{c}\in\mathbb{R}^{N}$,
while scalars are represented by lowercase letters. We perform three
operations on CTDs: addition, inner product, and the application of
a $d$-dimensional operator. 
\begin{itemize}
\item When two CTDs are added together, all terms are joined into a single
list and simply re-indexed. In such a case the nominal separation
rank is the sum of the ranks of the components, i.e., if CTDs are
of the ranks $\tilde{r}$ and $\hat{r}$, the output has rank $\tilde{r}+\hat{r}$. 
\item The inner product of two tensors in CTD format, $\tilde{\mathbf{F}}$
and $\hat{\mathbf{F}}$, is defined as 
\[
\left\langle \tilde{\mathbf{F}},\hat{\mathbf{F}}\right\rangle =\sum_{\tilde{l}=1}^{\tilde{r}}\sum_{\hat{l}=1}^{\hat{r}}\tilde{s}_{\tilde{l}}\hat{s}_{\hat{l}}\left\langle \tilde{\mathbf{F}}_{1}^{\tilde{l}},\hat{\mathbf{F}}_{1}^{\hat{l}}\right\rangle \dots\left\langle \tilde{\mathbf{F}}_{d}^{\tilde{l}},\hat{\mathbf{F}}_{d}^{\hat{l}}\right\rangle ,
\]
where the inner product $\left\langle \cdot,\cdot\right\rangle $
operating on vectors is the standard vector dot product. 
\item When applying a $d$-dimensional operator to a tensor in CTD format,
we have 
\[
\mathbb{A}\mathbf{F}=\sum_{\hat{l}=1}^{r_{\mathbb{A}}}\sum_{\tilde{l}=1}^{r_{\mathbf{F}}}s_{\hat{l}}^{\mathbb{A}}s_{\tilde{l}}^{\mathbf{F}}\left(\mathbb{A}_{1}^{\hat{l}}\mathbf{F}_{1}^{\tilde{l}}\right)\circ\dots\circ\left(\mathbb{A}_{d}^{\hat{l}}\mathbf{F}_{d}^{\tilde{l}}\right).
\]

\end{itemize}
We use the symbol $\Vert\cdot\Vert$ to denote the standard spectral
norm for matrices, as well as the Frobenius norm for tensors, 
\[
\Vert\mathbf{F}\Vert=\left\langle \mathbf{F},\mathbf{F}\right\rangle ^{\frac{1}{2}},
\]
and $\Vert\cdot\Vert_{1}$ and $\Vert\cdot\Vert_{2}$ to denote the
standard Euclidean $\ell_{1}$ and $\ell_{2}$ vector norms. 

For analysis involving matrices we use three different types of multiplication
in addition to the standard matrix multiplication. The Hadamard, or
entry-wise, product of two matrices $A$ and $B$ is denoted by $A*B$.
The Kronecker product of two matrices $A\in\mathbb{R}^{N_{A}\times M_{A}}$
and $B\in\mathbb{R}^{N_{B}\times M_{B}}$, is denoted as $A\otimes B$,
\[
A\otimes B=\left[\begin{array}{ccc}
A\left(1,1\right)B & \dots & A\left(1,M_{A}\right)B\\
\vdots & \ddots & \vdots\\
A\left(N_{A},1\right)B & \dots & A\left(N_{A},M_{A}\right)B
\end{array}\right].
\]
 The final type of matrix product we use, the Khatri-Rao product of
two matrices $A\in\mathbb{R}^{N_{A}\times M}$ and $B\in\mathbb{R}^{N_{B}\times M}$,
is denoted by $A\odot B,$
\[
A\odot B=\left[\begin{array}{cccc}
A\left(:,1\right)\otimes B\left(:,1\right) & A\left(:,2\right)\otimes B\left(:,2\right) & \dots & A\left(:,N\right)\otimes B\left(:,M\right)\end{array}\right].
\]
We also frequently use the maximal and minimal (non-zero) singular
values of a matrix, denoted as $\sigma_{\max}$ and $\sigma_{\min}$,
respectively.

\subsection{ALS algorithm}

Operations on tensors in CTD format lead to an increase of the nominal
separation rank. This separation rank is not necessarily the smallest
possible rank to represent the resulting tensor for a given accuracy.
The ALS algorithm attempts to find an approximation to the tensor
with minimal (or near minimal) separation rank. Specifically, given
a tensor $\mathbf{G}$ in CTD format with separation rank $r{}_{\mathbf{G}}$,
\[
\mathbf{G}=\sum_{l=1}^{r_{\mathbf{G}}}s_{l}^{\mathbf{G}}\mathbf{G}_{1}^{l}\circ\dots\circ\mathbf{G}_{d}^{l},
\]
and an acceptable error $\epsilon$, we attempt to find a representation
\[
\mathbf{F}=\sum_{\tilde{l}=1}^{r_{\mathbf{F}}}s_{\tilde{l}}^{\mathbf{F}}\mathbf{F}_{1}^{\tilde{l}}\circ\dots\circ\mathbf{F}_{d}^{\tilde{l}}
\]
with lower separation rank, $r_{\mathbf{F}}<r_{\mathbf{G}}$, such
that $\left\Vert \mathbf{F}-\mathbf{G}\right\Vert /\left\Vert \mathbf{G}\right\Vert <\epsilon$.

The standard ALS algorithm starts from an initial guess, $\mathbf{F}$,
with a small separation rank, e.g., $r_{\mathbf{F}}=1$. A sequence
of least squares problems in each direction is then constructed and
solved to update the representation. Given a direction $k$, we freeze
the factors in all other directions to produce a least squares problem
for the factors in direction $k$. This process is then repeated for
all directions $k$. One cycle through all $k$ is called an ALS sweep.
These ALS sweeps continue until the improvement in the residual $\left\Vert \mathbf{F}-\mathbf{G}\right\Vert /\left\Vert \mathbf{G}\right\Vert $
either drops below a certain threshold or reaches the desired accuracy,
i.e. $\left\Vert \mathbf{F}-\mathbf{G}\right\Vert /\left\Vert \mathbf{G}\right\Vert <\epsilon$.
If the residual is still above the target accuracy $\epsilon$, the
separation rank $r_{\mathbf{F}}$ is increased and we repeat the previous
steps for constructing the representation with the new separation
rank.

Specifically, as discussed in \cite{BEY-MOH:2005}, the construction
of the normal equations for direction $k$ can be thought of as taking
the derivatives of the Frobenius norm of $\Vert\mathbf{F}-\mathbf{G}\Vert^{2}$
with respect to the factors $\mathbf{F}_{k}^{\tilde{l}}$, ${\tilde{l}}=1,\dots,r_{\mathbf{F}}$,
and setting these derivatives to zero. This yields the normal equations
\begin{equation}
B_{k}\,\mathbf{c}_{j_{k}}=\mathbf{b}_{j_{k}},\label{eq:normal-equations-1}
\end{equation}
where $j_{k}$ corresponds to the $j$-th entry of $\mathbf{F}_{k}^{\tilde{l}}$
and $\mathbf{c}_{j_{k}}=c_{j_{k}}(\tilde{l})$ is a vector indexed
by $\tilde{l}$. Alternatively, the normal system (\ref{eq:normal-equations-1})
can be obtained by contracting all directions except the optimization
direction $k$, so that the entries of the matrix $B_{k}$ are the
Hadamard product of Gram matrices, 
\begin{equation}
B_{k}(\hat{l},\tilde{l})=\prod_{i\ne k}\left\langle \mathbf{F}_{i}^{\tilde{l}},\mathbf{F}_{i}^{\hat{l}}\right\rangle ,\label{eq:B-matrix-definition}
\end{equation}
and, accordingly, the right-hand side is 
\[
\mathbf{b}_{j_{k}}(\hat{l})=\sum_{l=1}^{r_{\mathbf{G}}}s_{l}^{\mathbf{G}}G_{k}^{l}\left(j_{k}\right)\prod_{i\ne k}\left\langle \mathbf{G}_{i}^{l},\mathbf{F}_{i}^{\hat{l}}\right\rangle .
\]
We solve \eqref{eq:normal-equations-1} for $\mathbf{c}_{j_{k}}$
and use the solution to update $\mathbf{F}_{k}^{\tilde{l}}$. Pseudocode
for the ALS algorithm is provided in Algorithm~\ref{alg:Alternating-least-squares},
where $max\_rank$ and $max\_iter$ denote the maximum separation
rank and the limit on the number of iterations. The threshold $\delta$
is used to decide if the separation rank needs to be increased.

\RestyleAlgo{boxruled}
\SetKwInOut{Input}{input}
\SetKwInOut{Output}{output}
\SetKw{Initialize}{initialize}
\SetKw{Return}{return}
\SetKw{Break}{break}

\begin{algorithm}
\Input{$\epsilon>0,\;\delta>0,\;\mathbf{G}\;\textrm{with rank}\; r_{\mathbf{G}},\; max\_rank,\; max\_iter$}

\Initialize $r_{\mathbf{F}}=1$ tensor $\mathbf{F}=\mathbf{F}_{1}^{1}\circ\dots\circ\mathbf{F}_{d}^{1}$
with randomly generated $\mathbf{F}_{k}^{1}$

\While{$r_{\mathbf{F}}\le max\_rank$} { $iter=1$\\
\If{$r_{\mathbf{F}}>1$}{add a random rank $1$ contribution to
$\mathbf{F}$: $\mathbf{F}=\mathbf{F}+\mathbf{F}_{1}^{r_{\mathbf{F}}}\circ\dots\circ\mathbf{F}_{d}^{r_{\mathbf{F}}}$}

$res=\left\Vert \mathbf{F}-\mathbf{G}\right\Vert /\left\Vert \mathbf{G}\right\Vert $

\While{$iter\le max\_iter$} { $res\_old=res$\\
 \For{$k=1,\dots,d$} { solve $B_{k}\mathbf{c}_{j_{k}}=\mathbf{b}_{j_{k}}$
for every $j_{k}$ in direction k\\
 define $\mathbf{v}_{\tilde{l}}=\left(c_{1}(\tilde{l}),\dots,c_{M_{k}}(\tilde{l})\right)$
for $\tilde{l}=1,\dots,r_{\mathbf{F}}$\\
 $s_{\tilde{l}}^{\mathbf{F}}=\left\Vert \mathbf{v}_{\tilde{l}}\right\Vert _{2}$
for $\tilde{l}=1,\dots,r_{\mathbf{F}}$\\
 $F_{k}^{\tilde{l}}(j_{k})=c_{j_{k}}(\tilde{l})/s_{\tilde{l}}^{\mathbf{F}}$
for $\tilde{l}=1,\dots,r_{\mathbf{F}}$ }

$res=\left\Vert \mathbf{F}-\mathbf{G}\right\Vert /\left\Vert \mathbf{G}\right\Vert $

\uIf{$res<\epsilon$}{\Return {$\mathbf{F}$}}\uElseIf{$\left|res-res\_old\right|<\delta$}{\Break}\Else{$iter=iter+1$}}

\protect\protect

$r_{\mathbf{F}}=r_{\mathbf{F}}+1$}

\protect

\Return { $\mathbf{F}$}

~

\protect\caption{Alternating least squares algorithm for rank reduction\label{alg:Alternating-least-squares}}
\end{algorithm}

A potential pitfall of the ALS algorithm is poor-conditioning of the
matrix $B_{k}$ since the construction of normal equations squares
the condition number as is well known in matrix problems.  An alternative
that avoids the normal equations is mentioned in the review paper
\cite{KOL-BAD:2009}, but it is not feasible for problems with even
moderately large dimension (e.g. $d=5$).

\subsection{Estimate of condition numbers of least squares matrices}

It is an empirical observation that the condition number of the matrices
$B_{k}$ is sometimes significantly better than the condition numbers
of some of the Gram matrices comprising the Hadamard product in (\ref{eq:B-matrix-definition}).
In fact we have 
\begin{lem}
\label{lem:formal estimate}Let $A$ and $B$ be Gram matrices with
all diagonal entries equal to $1$. Then we have 
\[
\sigma_{\min}\left(B\right)\le\sigma_{\min}\left(A*B\right)\le\sigma_{\max}\left(A*B\right)\le\sigma_{\max}\left(B\right).
\]
If the matrix $B$ is positive definite, then 
\[
\kappa\left(A*B\right)\le\kappa\left(B\right).
\]

\end{lem}
Since Gram matrices are symmetric non-negative definite, the proof
of Lemma~\ref{lem:formal estimate} follows directly from \cite[Theorem 5.3.4]{HOR-JOH:1994}.
This estimate implies that it is sufficient for only one of the matrices
to be well conditioned to assure that the Hadamard product is also
well conditioned. In other words, it is necessary for all directional
Gram matrices to be ill-conditioned to cause the ill-conditioning
of the Hadamard product. Clearly, this situation can occur and we
address it in the paper.

\subsection{Modification of normal equations: motivation for randomized methods }

We motivate our approach by first considering an alternative to forming
normal equations for ordinary matrices (excluding the QR factorization
that can be easily used for matrices). Given a matrix $A\in\mathbb{R}^{N\times n}$,
$N\ge n$, we can multiply $A\mathbf{x}=\mathbf{b}$ by a matrix $R\in\mathbb{R}^{n^{\prime}\times N}$
with independent random entries and then solve 
\begin{equation}
RA\mathbf{x}=R\mathbf{b},\label{eq:random-equations}
\end{equation}
instead (see, e.g. \cite{HA-MA-TR:2011,ROK-TYG:2008,SARLOS:2006,W-L-R-T:2008}
). The solution of this system, given that $R$ is of appropriate
size (i.e., $n^{\prime}$ is large enough), will be close to the least
squares solution \cite[Lemma 2]{ROK-TYG:2008}. In \cite{ROK-TYG:2008},
\eqref{eq:random-equations} is used to form a preconditioner and
an initial guess for solving $\min~\Vert A\mathbf{x}-\mathbf{b}\Vert_{2}$
via a preconditioned conjugate gradient method. However, for our application
we are interested in using equations of the form \eqref{eq:random-equations}
in the Hadamard product in \eqref{eq:B-matrix-definition}. We observe
that $RA$ typically has a smaller condition number than $A^{T}A$.
To see why, recall that for full-rank, square matrices $A$ and $B$,
a bound on the condition number is 
\[
\kappa(AB)\le\kappa(A)\kappa(B).
\]
However, for rectangular full-rank matrices $A\in\mathbb{R}^{r^{\prime}\times N}$
and $B\in\mathbb{R}^{N\times r}$, $r\le r^{\prime}\le N$, this inequality
does not necessarily hold. Instead, we have the inequality 
\begin{equation}
\kappa(AB)\le\kappa(A)\frac{\sigma_{1}\left(B\right)}{\sigma_{\min}\left(P_{A^{T}}\left(B\right)\right)},\label{eq:rec-cond-num-ineq}
\end{equation}
where $P_{A^{T}}\left(B\right)$ is the projection of $B$ onto the
row space of $A$ (for a proof of this inequality, see Appendix~A).
If $A$ has a small condition number (for example, when \textbf{$A$}
is a Gaussian random matrix, see \cite{CHE-DON:2005,EDELMA:1988,EDELMA:1989})
and we were to assume $\sigma_{\min}\left(P_{A^{T}}\left(B\right)\right)$
is close to $\sigma_{\min}\left(B\right)$, we obtain condition numbers
smaller than $\kappa^{2}(B)$. The assumption that $\sigma_{\min}\left(P_{A^{T}}\left(B\right)\right)$
is close to $\sigma_{\min}\left(B\right)$ is the same as assuming
the columns of $B$ lie within the subspace spanned by the row of
$A$. This is achieved by choosing $r^{\prime}$ to be larger than
$r$ when $A$ is a randomized matrix.

\subsection{Definitions and random matrix theory}

The main advantage of our approach is an improved condition number
for the linear system solved at every step of the ALS algorithm. We
use a particular type of random matrices to derive bounds on the condition
number: the rows are independently distributed random vectors, but
the columns are not (instead of the standard case where all entries
are i.i.d). Such matrices were studied extensively by Vershynin \cite{VERSHY:2012}
and we rely heavily on this work for our estimates. To proceed, we
need the following definitions from \cite{VERSHY:2012}.
\begin{rem}
Definitions involving random variables, and vectors composed of random
variables, are not consistent with the notation of the rest of the
paper, outlined in Section~\ref{sub:Notation}.\end{rem}
\begin{defn}
\cite[Definition 5.7]{VERSHY:2012} Let $\mathbb{P}\{\cdot\}$ denote
the probability of a set and $\mathbb{E}$ the mathematical expectation
operator. Also, let $X$ be a random variable that satisfies one of
the three following equivalent properties,
\begin{flalign*}
1. & \,\,\mathbb{P}\left\{ \left|X\right|>t\right\} \le\exp\left(1-t^{2}/K_{1}^{2}\right)\,\,\textrm{for all}\,\, t\ge0\\
2. & \,\,\left(\mathbb{E}\left|X\right|^{p}\right)^{1/p}\le K_{2}\sqrt{p}\,\,\textrm{for all}\,\, p\ge1\\
3. & \,\,\mathbb{E}\exp\left(X^{2}/K_{3}^{2}\right)\le e,
\end{flalign*}
where the constants $K_{i}$, $i=1,2,3$, differ from each other by
at most an absolute constant factor (see \cite[Lemma 5.5]{VERSHY:2012}
for a proof of the equivalence of these properties). Then $X$ is
called a sub-Gaussian random variable. The sub-Gaussian norm of $X$
is defined as the smallest $K_{2}$ in property~2, i.e., 
\[
\left\Vert X\right\Vert _{\psi_{2}}=\underset{p\ge1}{\sup}\frac{\left(\mathbb{E}\left|X\right|^{p}\right)^{1/p}}{\sqrt{p}}.
\]

\end{defn}
\noindent Examples of sub-Gaussian random variables include Gaussian
and Bernoulli random variables. We also present definitions for sub-Gaussian
random vectors and their norm. 
\begin{defn}
\cite[Definition 5.7]{VERSHY:2012} A random vector $X\in\mathbb{R}^{n}$
is called a sub-Gaussian random vector if $\left\langle X,\mathbf{x}\right\rangle $
is a sub-Gaussian random variable for all $\mathbf{x}\in\mathbb{R}^{n}$.
The sub-Gaussian norm of $X$ is subsequently defined as 
\[
\left\Vert X\right\Vert _{\psi_{2}}=\underset{\mathbf{x}\in\mathcal{S}^{n-1}}{\sup}\left\Vert \left\langle X,\mathbf{x}\right\rangle \right\Vert _{\psi_{2}},
\]
where $\mathcal{S}^{n-1}$ is the unit Euclidean sphere.
\end{defn}
\begin{defn}
\cite[Definition 5.19]{VERSHY:2012} A random vector $X\in\mathbb{R}^{n}$
is called isotropic if its second moment matrix, $\Sigma=\Sigma\left(X\right)=\mathbb{E}\left[XX^{T}\right]$,
is equal to identity, i.e. $\Sigma\left(X\right)=I$. This definition
is equivalent to 
\[
\mathbb{E}\left\langle X,\mathbf{x}\right\rangle ^{2}=\left\Vert \mathbf{x}\right\Vert _{2}^{2}\,\,\,\,\textrm{for all}\,\,\mathbf{x}\in\mathbb{R}^{n}.
\]

\end{defn}
The following theorem from \cite{VERSHY:2012} provides bounds on
the condition numbers of matrices whose rows are independent sub-Gaussian
isotropic random variables. 
\begin{thm}
\label{thm:Vershynin}\cite[Theorem 5.38]{VERSHY:2012} Let $A$ be
an $N\times n$ matrix whose rows $A\left(i,:\right)$ are independent,
sub-Gaussian isotropic random vectors in $\mathbb{R}^{n}$. Then for
every $t\ge0$, with probability at least $1-2\,\exp\left(-ct^{2}\right)$,
one has 
\begin{equation}
\sqrt{N}-C\sqrt{n}-t\le\sigma_{\min}\left(A\right)\le\sigma_{\max}\left(A\right)\le\sqrt{N}+C\sqrt{n}+t.\label{eq:Vershynin-theorem-bounds}
\end{equation}
Here $C=C_{K}$, $c=c_{K}>0$, depend only on the sub-Gaussian norm
$K=\underset{i}{\max}\left\Vert A\left(i,:\right)\right\Vert _{\psi_{2}}.$ 
\end{thm}
An outline of the proof of Theorem~\ref{thm:Vershynin} will be useful
for deriving our own results, so we provide a sketch in Appendix~A.
The following lemma is used to prove Theorem~\ref{thm:Vershynin},
and will also be useful later on in the paper. We later modify it
to prove a version of Theorem~\ref{thm:Vershynin} that works for
sub-Gaussian, non-isotropic random vectors. 
\begin{lem}
\label{lem:Vershynin}\cite[Lemma 5.36]{VERSHY:2012} Consider a matrix
$B$ that satisfies 
\[
\left\Vert B^{T}B-I\right\Vert <\max\left(\delta,\delta^{2}\right)
\]
for some $\delta>0.$ Then 
\begin{equation}
1-\delta\le\sigma_{\min}\left(B\right)\le\sigma_{\max}\left(B\right)\le1+\delta.\label{eq:vershynin-SV-bounds}
\end{equation}
Conversely, if $B$ satisfies (\ref{eq:vershynin-SV-bounds}) for
some $\delta>0,$ then $\left\Vert B^{T}B-I\right\Vert <3\,\max\left(\delta,\delta^{2}\right)$.
\end{lem}

\section{Randomized ALS algorithm\label{sec:Randomized-ALS-algorithm}}

\subsection{Alternating least squares algorithm using random matrices}

We propose the following alternative to using the normal equations
in ALS algorithms: instead of (\ref{eq:B-matrix-definition}), define
the entries of $B_{k}$ via randomized projections,
\begin{equation}
B_{k}(\hat{l},\tilde{l})=\prod_{i\ne k}\left\langle \mathbf{F}_{i}^{\tilde{l}},\mathbf{R}_{i}^{\hat{l}}\right\rangle ,\label{eq:randomized-ALS-B}
\end{equation}
where $\mathbf{R}_{i}^{\hat{l}}$ is the $\hat{l}$-th column of a
matrix $R_{i}\in\mathbb{R}^{M_{i}\times r'}$, $r'>r$, with random
entries corresponding to direction $i$. The choice of $r'>r$ is
made to reduce the condition number of $B_{k}$. As shown in Section~\ref{sub:Bounding-the-condition-number},
as $r/r^{\prime}\rightarrow0$ the bound on $\kappa\left(B_{k}\right)$
goes to $\kappa\left(B_{k}\right)\le\kappa\left(\left(B_{k}^{ALS}\right)^{\frac{1}{2}}\right)$,
where $B_{k}^{ALS}$ is the $B_{k}$ matrix for standard ALS, i.e.
\eqref{eq:B-matrix-definition}. In this paper we consider independent
signed Bernoulli random variables, i.e., $R_{i}(j_{k},\hat{l})$ is
either $-1$ or $1$ each with probability $1/2$. We have had some
success using standard Gaussian random variables in our experiments
as well. The proposed change also alters the right-hand side of the
normal equations (\ref{eq:normal-equations-1}), 
\begin{equation}
\mathbf{b}_{j_{k}}(\hat{l})=\sum_{l=1}^{r_{\mathbf{G}}}s_{l}^{\mathbf{G}}G_{k}^{l}\left(j_{k}\right)\prod_{i\ne k}\left\langle \mathbf{G}_{i}^{l},\mathbf{R}_{i}^{\hat{l}}\right\rangle .\label{eq:RHS-rand-tensor}
\end{equation}
Equivalently, $B_{k}$ may be written as 
\[
B_{k}=\prod_{i\ne k}R_{i}^{T}F_{i}.
\]
Looking ahead, we choose random matrices $R_{i}$ such that $B_{k}$
is a tall, rectangular matrix. Solving the linear system (\ref{eq:normal-equations-1})
with rectangular $B_{k}$ will require a pseudo-inverse, computed
via either the singular value decomposition (SVD) or a QR algorithm. 

To further contrast the randomized ALS algorithm with the standard
ALS algorithm, we highlight two differences: firstly, the randomized
ALS trades the monotonic reduction of approximation error (a property
of the standard ALS algorithm) for better conditioning. To adjust
we use a simple tactic: if a randomized ALS sweep (over all directions)
decreases the error, we keep the resulting approximation. Otherwise,
we discard the sweep, generate independent random matrices $R_{i}$,
and rerun the sweep. Secondly, the randomized ALS algorithm can be
more computationally expensive than the standard one. This is due
to the rejection scheme outlined above and the fact that $B_{k}$
in the randomized algorithm has a larger number of rows than its standard
counterpart, i.e., $r'>r$. Pseudocode of our new algorithm is presented
in Algorithm~\ref{alg:Randomized-alternating-least-squares}.

\begin{algorithm}[h]
\Input{$\epsilon>0,\;\mathbf{G}\;\textrm{with rank}\; r_{\mathbf{G}},\; max\_tries,\; max\_rank,\; max\_iter$}

\Initialize $r_{\mathbf{F}}=1$ tensor $\mathbf{F}=\mathbf{F}_{1}^{1}\circ\dots\circ\mathbf{F}_{d}^{1}$
with randomly generated $\mathbf{F}_{k}^{1}$\\

\While{$r_{\mathbf{F}}\le max\_rank$} {$tries=1$

$iter=1$

construct randomized tensor $\mathbf{R}$

\If{$r_{\mathbf{F}}>1$}{add a random rank $1$ contribution to
$\mathbf{F}$: $\mathbf{F}=\mathbf{F}+\mathbf{F}_{1}^{r_{\mathbf{F}}}\circ\dots\circ\mathbf{F}_{d}^{r_{\mathbf{F}}}$}

\While{$iter\le max\_iter$ and $tries\le max\_tries$} { $\mathbf{F}_{old}=\mathbf{F}$\\
\For{$k=1,\dots,d$} { construct $B_{k}$, using \eqref{eq:randomized-ALS-B}\\
solve $B_{k}\mathbf{c}_{j_{k}}=\mathbf{b}_{j_{k}}$ for every $j_{k}$
in direction k\\
define $\mathbf{v}_{\tilde{l}}=\left(c_{1}(\tilde{l}),\dots,c_{M_{k}}(\tilde{l})\right)$
for $\tilde{l}=1,\dots,r_{\mathbf{F}}$\\
$s_{\tilde{l}}^{\mathbf{F}}=\left\Vert \mathbf{v}_{\tilde{l}}\right\Vert _{2}$
for $\tilde{l}=1,\dots,r_{\mathbf{F}}$\\
$F_{k}^{\tilde{l}}(j_{k})=c_{j_{k}}(\tilde{l})/s_{\tilde{l}}^{\mathbf{F}}$
for $\tilde{l}=1,\dots,r_{\mathbf{F}}$ } \uIf{$\left\Vert \mathbf{F}-\mathbf{G}\right\Vert /\left\Vert \mathbf{G}\right\Vert <\epsilon$}
{\Return{ $\mathbf{F}$} }\uElseIf{$\left\Vert \mathbf{F}_{old}-\mathbf{G}\right\Vert /\left\Vert \mathbf{G}\right\Vert <\left\Vert \mathbf{F}-\mathbf{G}\right\Vert /\left\Vert \mathbf{G}\right\Vert $}
{$\mathbf{F}=\mathbf{F}_{old}$ \\
$tries=tries+1$\\
$iter=iter+1$} \Else {$tries=1$\\
$iter=iter+1$} } 

\protect\protect

$r_{\mathbf{F}}=r_{\mathbf{F}}+1$}

\protect

\Return {$\mathbf{F}$}

~

\protect\caption{Randomized alternating least squares algorithm for rank reduction\label{alg:Randomized-alternating-least-squares}}
\end{algorithm}

\begin{rem}
\label{rem:precond-approach}We have explored an alternative approach
using projections onto random tensors, different from Algorithm~\ref{alg:Randomized-alternating-least-squares}.
Instead of using $B_{k}$ in \eqref{eq:randomized-ALS-B} to solve
for $\mathbf{c}_{j_{k}}$, we use the QR factorization of $B_{k}$
to form a preconditioner matrix, similar to the approach of \cite{ROK-TYG:2008}
for solving overdetermined least squares problems in a matrix setting.
This preconditioner is used to improve the condition number of $B_{k}$
in \eqref{eq:B-matrix-definition}. The approach is different from
Algorithm~\ref{alg:Randomized-alternating-least-squares}: we solve
the same equations as the standard ALS algorithm, but in a better
conditioned manner. Solving the same equations preserves the monotone
error reduction property of standard ALS. With Algorithm~\ref{alg:Randomized-alternating-least-squares}
the equations we solve are different, but, as shown in Section~\eqref{sub:Convergence},
the solutions of each least squares problem are close to the those
obtained by the standard ALS algorithm.
\end{rem}
We provide convergence results and theoretical bounds on the condition
number of $B_{k}$ with entries (\ref{eq:randomized-ALS-B}) in Sections~\ref{sub:Convergence}
and~\ref{sub:Bounding-the-condition-number}, respectively. Additionally,
in Section~\ref{sec:Examples}, we empirically demonstrate the superior
conditioning properties of $B_{k}$ defined in (\ref{eq:randomized-ALS-B})
relative to those given by the standard ALS in (\ref{eq:B-matrix-definition}).

\subsection{Convergence of the randomized ALS algorithm\label{sub:Convergence} }

Before deriving bounds on the condition number of (\ref{eq:randomized-ALS-B}),
it is important to discuss the convergence properties of our algorithm.
To do so for our tensor algorithm, we derive a convergence result
similar to \cite[Lemma 4.8]{W-L-R-T:2008}. In this analysis, we flatten
our tensors into large matrices and use results from random matrix
theory to show convergence. First, we construct the large matrices
used in this section from (\ref{eq:randomized-ALS-B}). Writing the
inner product as a sum allows us to group all the summations together, 

We have
\begin{eqnarray*}
B_{k}\left(\hat{l},\tilde{l}\right) & = & \prod_{i\ne k}\sum_{j_{i}=1}^{M_{i}}F_{i}^{\tilde{l}}\left(j_{i}\right)R_{i}^{\hat{l}}\left(j_{i}\right)\\
 & = & \sum_{j_{1}=1}^{M_{1}}\dots\sum_{j_{k-1}=1}^{M_{k-1}}\sum_{j_{k+1}=1}^{M_{k+1}}\dots\sum_{j_{d}=1}^{M_{d}}\left(F_{1}^{\tilde{l}}\left(j_{1}\right)\dots\right)\left(R_{1}^{\hat{l}}\left(j_{1}\right)\dots\right),
\end{eqnarray*}
where we have expanded the product to get the sum of the products
of individual entries. Introducing a multi-index $j=\left(j_{1},\dots,j_{k-1},j_{k+1},\dots,j_{d}\right)$,
we define two matrices, $A_{k}\in\mathbb{R}^{M\times r}$ and $R_{k}\in\mathbb{R}^{M\times r^{\prime}}$,
where $M=\prod_{i\ne k}M_{i}$ is large, i.e., we write 
\begin{eqnarray*}
A_{k}\left(j,\tilde{l}\right) & = & F_{1}^{\tilde{l}}\left(j_{1}\right)\dots F_{k-1}^{\tilde{l}}\left(j_{k-1}\right)F_{k+1}^{\tilde{l}}\left(j_{k+1}\right)\dots F_{d}^{\tilde{l}}\left(j_{d}\right)\\
R_{k}\left(j,\hat{l}\right) & = & R_{1}^{\hat{l}}\left(j_{1}\right)\dots R_{k-1}^{\hat{l}}\left(j_{k-1}\right)R_{k+1}^{\hat{l}}\left(j_{k+1}\right)\dots R_{d}^{\hat{l}}\left(j_{d}\right).
\end{eqnarray*}
We note that these matrices can also be written as Khatri-Rao products,
\begin{eqnarray}
A_{k} & = & F_{1}\odot\dots\odot F_{k-1}\odot F_{k+1}\odot\dots\odot F_{d}\nonumber \\
R_{k} & = & R_{1}\odot\dots\odot R_{k-1}\odot R_{k+1}\odot\dots\odot R_{d}.\label{eq:Khatri-Rao-defintions}
\end{eqnarray}
Since $M$ is large, $M\gg r'>r$, both $A$ and $R$ are rectangular
matrices. Similarly, we rewrite a vector $\mathbf{b}$ in \eqref{eq:RHS-rand-tensor},
\begin{eqnarray*}
\mathbf{b}_{j_{k}}(\hat{l}) & = & \sum_{l=1}^{r_{\mathbf{G}}}s_{l}^{\mathbf{G}}G_{k}^{l}\left(j_{k}\right)\sum_{j_{1}=1}^{M_{1}}\dots\sum_{j_{k-1}=1}^{M_{k-1}}\sum_{j_{k+1}=1}^{M_{k+1}}\dots\sum_{j_{d}=1}^{M_{d}}\left(G_{1}^{l}\left(j_{1}\right)\dots\right)\left(R_{1}^{\hat{l}}\left(j_{1}\right)\dots\right),
\end{eqnarray*}
using the multi-index $j$ as 
\[
\mathbf{b}_{k}\left(j\right)=\sum_{l=1}^{r_{\mathbf{G}}}s_{l}^{\mathbf{G}}G_{k}^{l}\left(j_{k}\right)\left(G_{1}^{l}\left(j_{1}\right)\dots G_{k-1}^{l}\left(j_{k-1}\right)G_{k+1}^{l}\left(j_{k+1}\right)\dots G_{d}^{l}\left(j_{d}\right)\right).
\]
Using the introduced notation, $A_{k},\, R_{k},$ and $\mathbf{b}_{k}$,
we rewrite the normal equations (\ref{eq:normal-equations-1}) for
direction $k$ and coordinate $j_{k}$ as 
\begin{equation}
A_{k}^{T}A_{k}\mathbf{c}_{k}=A_{k}^{T}\mathbf{b},\label{eq:orig_ls}
\end{equation}
and the randomized version of those equations as 
\begin{equation}
R_{k}^{T}A_{k}\mathbf{c}_{k}=R_{k}^{T}\mathbf{b}_{k}.\label{eq:random_ls}
\end{equation}

We highlight the notable difference between the random matrix $R_{k}$
above and those found in the usual matrix settings, for instance,
in randomized least squares regression \cite{HA-MA-TR:2011,ROK-TYG:2008,SARLOS:2006,W-L-R-T:2008}.
Specifically, in the former, the entries of $R_{k}$ are not statistically
independent and are products of random variables, whereas in the latter
the entries are often i.i.d realizations of single random variables.
In the present work, we utilize signed Bernoulli random variables,
where the entries of $R_{k}$ are dependent but also follow a signed
Bernoulli distribution. Whether there exists an optimal choice of
distribution for setting the entries $R{}_{k}$ in the tensor case
requires a careful examination which is beyond the scope of this paper.

Next, we present a convergence result showing that the solution to
the least squares problem at each iteration of randomized ALS is close
to the solution we would get using standard ALS. For ease of notation,
we drop the subscript $k$ from $A$ and $R$.
\begin{lem}
\label{lem:conv-bound-with-cond-number}Given $A\in\mathbb{R}^{M\times r}$
and $R\in\mathbb{R}^{M\times r^{\prime}}$ where $r\le r^{\prime}\le M$,
and assuming that $\mathbf{x}\in\mathbb{R}^{r}$ is the solution that
minimizes $\left\Vert R^{T}A\hat{\mathbf{x}}-R^{T}\mathbf{b}\right\Vert $$_{2}$
and $\mathbf{y}\in\mathbb{R}^{r}$ is the solution that minimizes
$\Vert A\hat{\mathbf{y}}-\mathbf{b}\Vert_{2}$, then 
\begin{equation}
\Vert A\mathbf{x}-\mathbf{b}\Vert_{2}\le\kappa\left(R^{T}Q\right)\,\left\Vert A\mathbf{y}-\mathbf{b}\right\Vert _{2},\label{eq:desired-inequality}
\end{equation}
where $Q\in\mathbb{R}^{M\times r_{Q}}$, $r_{Q}\le r+1$, is a matrix
with orthonormal columns from the QR factorization of the augmented
matrix \textup{$\left[A\,\brokenvert\,\mathbf{b}\right]$ }and where
$R^{T}Q$ is assumed to have full rank.
\end{lem}

\begin{proof}
We form the augmented matrix $\left[A\,\brokenvert\,\mathbf{b}\right]$
and find its QR decomposition, $\left[A\,\brokenvert\,\mathbf{b}\right]=QT$,
where $T=\left[T_{A}\,\brokenvert\, T_{\mathbf{b}}\right]$, $T_{A}\in\mathbb{R}^{r_{Q}\times r}$
and $T_{\mathbf{b}}\in\mathbb{R}^{r_{Q}}$, and $Q\in\mathbb{R}^{M\times r_{Q}}$
has orthonormal columns. Therefore, we have 
\begin{eqnarray*}
A & = & QT_{A}\\
\mathbf{b} & = & QT_{\mathbf{b}}.
\end{eqnarray*}
Using these decompositions of $A$ and $\mathbf{b}$, we define a
matrix $\Theta$ such that 
\begin{eqnarray*}
\Theta R^{T}A & = & A\\
\Theta R^{T}\mathbf{b} & = & \mathbf{b},
\end{eqnarray*}
and arrive at 
\[
\Theta=Q\left(\left(R^{T}Q\right)^{T}\left(R^{T}Q\right)\right)^{-1}\left(RQ\right)^{T},
\]
assuming that $\left(R^{T}Q\right)^{T}\left(R^{T}Q\right)$ is invertible.
We address this additional assumption when discussing the bounds of
the extreme singular values of $R^{T}Q$ in Theorem~\ref{thm:conv-thrm}.

Starting from the left-hand side of (\ref{eq:desired-inequality}),
we have 
\begin{eqnarray*}
\left\Vert A\mathbf{x}-\mathbf{b}\right\Vert _{2} & = & \left\Vert \Theta R^{T}A\mathbf{x}-\Theta R^{T}\mathbf{b}\right\Vert _{2}\\
 & \le & \left\Vert \Theta\right\Vert \left\Vert R^{T}A\mathbf{x}-R^{T}\mathbf{b}\right\Vert _{2}\\
 & \le & \left\Vert \Theta\right\Vert \left\Vert R^{T}A\mathbf{y}-R^{T}\mathbf{b}\right\Vert _{2}.
\end{eqnarray*}
Since multiplication by $A$ maps a vector to the column space of
$A$, there exists $T_{\mathbf{y}}\in\mathbb{R}^{r_{Q}}$ such that
$A\mathbf{y}=QT_{\mathbf{y}}$. Hence, we obtain 
\begin{eqnarray*}
\left\Vert A\mathbf{x}-b\right\Vert _{2} & \le & \left\Vert \Theta\right\Vert \left\Vert R^{T}QT_{\mathbf{y}}-R^{T}QT_{\mathbf{b}}\right\Vert _{2}\\
 & \le & \left\Vert \Theta\right\Vert \left\Vert R^{T}Q\right\Vert \left\Vert T_{\mathbf{y}}-T_{\mathbf{b}}\right\Vert _{2}\\
 & \le & \left\Vert \Theta\right\Vert \left\Vert R^{T}Q\right\Vert \left\Vert A\mathbf{y}-\mathbf{b}\right\Vert _{2},
\end{eqnarray*}
where in the last step we used the orthonormality of the columns of
$Q$.

Next we estimate norms, $\left\Vert \Theta\right\Vert $ and $\left\Vert R^{T}Q\right\Vert $.
First, we decompose $R^{T}Q$ using the singular value decomposition,
$R^{T}Q=U\Sigma V^{T}$. From the definition of the spectral norm
we know $\left\Vert R^{T}Q\right\Vert =\sigma_{\max}\left(R^{T}Q\right)$.
Using the SVD of $R^{T}Q$ and the definition of $\Theta$ along with
our assumption of invertibility of $\left(R^{T}Q\right)^{T}\left(R^{T}Q\right)$
, we write 
\begin{eqnarray*}
\left\Vert \Theta\right\Vert  & = & \left\Vert Q\left(\left(R^{T}Q\right)^{T}\left(R^{T}Q\right)\right)^{-1}\left(R^{T}Q\right)^{T}\right\Vert \\
 & = & \left\Vert \left(V\Sigma^{2}V^{T}\right)^{-1}V\Sigma U^{T}\right\Vert \\
 & = & \left\Vert V\Sigma^{-2}V^{T}V\Sigma U^{T}\right\Vert \\
 & = & \left\Vert V\Sigma^{-1}U^{T}\right\Vert .
\end{eqnarray*}
Hence $\left\Vert \Theta\right\Vert =1/\sigma_{\min}(R^{T}Q)$, and
the bound is 
\begin{eqnarray*}
\left\Vert A\mathbf{x}-\mathbf{b}\right\Vert _{2} & \le & \sigma_{\max}\left(R^{T}Q\right)/\sigma_{\min}(R^{T}Q)\,\left\Vert A\mathbf{y}-\mathbf{b}\right\Vert _{2}\\
 & \le & \kappa\left(R^{T}Q\right)\,\left\Vert A\mathbf{y}-\mathbf{b}\right\Vert _{2}.
\end{eqnarray*}

\end{proof}
We later use results from \cite{VERSHY:2012} to bound $\kappa\left(R^{T}Q\right)$
since $R^{T}Q$ is a random matrix whose rows are independent from
one another, but whose columns are not. To use this machinery, specifically
Theorem~\ref{thm:Vershynin}, we require the following lemma. 
\begin{lem}
\label{lem:RQ-isotropic}$R^{T}Q$ is a random matrix with isotropic
rows.\end{lem}
\begin{proof}
Using the second moment matrix, we show that the rows of $R^{T}Q$
are isotropic. Given a row of $R^{T}Q$ written in column form, $\left[R\left(:,\hat{l}\right)^{T}Q\right]^{T}=Q^{T}R\left(:,\hat{l}\right)$,
we form the second moment matrix,
\[
\mathbb{E}\left[Q^{T}R\left(:,\hat{l}\right)R\left(\hat{l},:\right)^{T}Q\right]=Q^{T}\mathbb{E}\left[R\left(:,\hat{l}\right)R\left(\hat{l},:\right)^{T}\right]Q.
\]
 and show 
\[
\mathbb{E}\left[R\left(:,\hat{l}\right)R\left(:,\hat{l}\right)^{T}\right]=I_{M\times M}.
\]
Hence $\mathbb{E}\left[Q^{T}R\left(:,\hat{l}\right)R\left(\hat{l},:\right)^{T}Q\right]=Q^{T}Q=I_{r_{Q}\times r_{Q}}$
and $R^{T}Q$ is isotropic.%

From the Khatri-Rao product definition of the matrix $R$ \eqref{eq:Khatri-Rao-defintions},
we write a column of $R$ as
\[
R\left(:,\hat{l}\right)=\bigotimes_{\begin{array}{c}
i=1:d\\
i\ne k
\end{array}}R_{i}\left(:,\hat{l}\right).
\]
Therefore, using properties of the Kronecker product (see, e.g. \cite[equation (2.2)]{KOL-BAD:2009})
we can switch the order of the regular matrix product and the Kronecker
products,
\[
R\left(:,\hat{l}\right)R\left(:,\hat{l}\right)^{T}=\bigotimes_{\begin{array}{c}
i=1:d\\
i\ne k
\end{array}}R_{i}\left(:,\hat{l}\right)R_{i}\left(:,\hat{l}\right)^{T}.
\]
Taking the expectation and moving it inside the Kronecker product
gives us
\begin{eqnarray*}
\mathbb{E}\left[R\left(:,\hat{l}\right)R\left(:,\hat{l}\right)^{T}\right] & = & \bigotimes_{\begin{array}{c}
i=1:d\\
i\ne k
\end{array}}\mathbb{E}\left[R_{i}\left(:,\hat{l}\right)R_{i}\left(:,\hat{l}\right)^{T}\right]\\
 & = & \bigotimes_{\begin{array}{c}
i=1:d\\
i\ne k
\end{array}}I_{M_{i}\times M_{i}}\\
 & = & I_{M\times M}.
\end{eqnarray*}

\end{proof}
Since $R^{T}Q$ is a tall rectangular matrix ($R^{T}Q\in\mathbb{R}^{r^{\prime}\times r_{Q}}$)
with independent sub-Gaussian isotropic rows, we may use Theorem~5.39
from Vershynin to bound the extreme singular values. 
\begin{lem}
\label{lem:RQ bound}For every $t\ge0$, with probability at least
$1-2\exp\left(-ct^{2}\right)$ we have
\begin{equation}
\kappa\left(R^{T}Q\right)\le\frac{1+C\sqrt{\left(r+1\right)/r^{\prime}}+t/\sqrt{r^{\prime}}}{1-C\sqrt{\left(r+1\right)/r^{\prime}}-t/\sqrt{r^{\prime}}},\label{eq:bound-on-RtQ}
\end{equation}
where $C=C_{K}$ and $c=c_{K}>0$ depend only on the sub-Gaussian
norm $K=\max_{i}\left\Vert R\left(:,i\right)^{T}Q\right\Vert _{\psi_{2}}$
of the rows of $R^{T}Q$.\end{lem}
\begin{proof}
Using Lemma~\ref{lem:RQ-isotropic} and Theorem~\ref{thm:Vershynin},
we have the following bound on the extreme condition numbers of $R^{T}Q\in\mathbb{R}^{r^{\prime}\times r_{Q}}$
for every $t\ge0$, with probability at least $1-2\exp\left(-ct^{2}\right)$,

\[
\sqrt{r^{\prime}}-C\sqrt{r_{Q}}-t\le\sigma_{\min}\left(R^{T}Q\right)\le\sigma_{\max}\left(R^{T}Q\right)\le\sqrt{r^{\prime}}+C\sqrt{r_{Q}}+t,
\]
where $C=C_{K}$ and $c=c_{K}>0$ depend only on the sub-Gaussian
norm $K=\max_{i}\left\Vert \left(R^{T}Q\right)_{i}\right\Vert _{\psi_{2}}$
of the rows of $R^{T}Q$. Since $r_{Q}\le r+1$, we have 
\[
\sqrt{r^{\prime}}-C\sqrt{r+1}-t\le\sigma_{\min}\left(R^{T}Q\right)\le\sigma_{\max}\left(R^{T}Q\right)\le\sqrt{r^{\prime}}+C\sqrt{r+1}+t,
\]
with the same probability. 

We now state the convergence result.\end{proof}
\begin{thm}
\label{thm:conv-thrm} Given $A\in\mathbb{R}^{M\times r}$ and $R\in\mathbb{R}^{M\times r^{\prime}}$
where $r\le r^{\prime}\le M$, and assuming that $\mathbf{x}\in\mathbb{R}^{r}$
is the solution that minimizes $\left\Vert R^{T}A\hat{\mathbf{x}}-R^{T}\mathbf{b}\right\Vert $$_{2}$
and $\mathbf{y}\in\mathbb{R}^{r}$ is the solution that minimizes
$\Vert A\hat{\mathbf{y}}-\mathbf{b}\Vert_{2}$, then for every $t\ge0$,
with probability at least $1-2\exp\left(-ct^{2}\right)$ we have 
\[
\left\Vert A\mathbf{x}-\mathbf{b}\right\Vert \le\frac{1+C\sqrt{\left(r+1\right)/r^{\prime}}+t/\sqrt{r^{\prime}}}{1-C\sqrt{\left(r+1\right)/r^{\prime}}-t/\sqrt{r^{\prime}}}\,\left\Vert A\mathbf{y}-\mathbf{b}\right\Vert ,
\]
where $C=C_{K}$ and $c=c_{K}>0$ depend only on the sub-Gaussian
norm $K=\max_{i}\left\Vert R\left(:,i\right)^{T}Q\right\Vert _{\psi_{2}}$.
The matrix $Q\in\mathbb{R}^{M\times r_{Q}}$, $r_{Q}\le r+1$, is
composed of orthonormal columns from the QR factorization of the augmented
matrix \textup{$\left[A\,\brokenvert\,\mathbf{b}\right]$ }and $R^{T}Q$
is assumed to have full rank.\end{thm}
\begin{proof}
The proof results from bounding $\kappa\left(R^{T}Q\right)$ in Lemma~\ref{lem:conv-bound-with-cond-number}
with Lemma~\ref{lem:RQ bound} via Lemma~\ref{lem:RQ-isotropic}. 
\end{proof}

\subsection{\label{sub:Bounding-the-condition-number}Bounding the condition
number of $B_{k}$}

To bound the condition number of $B_{k}$, we use a modified version
of Theorem~\ref{thm:Vershynin}. If the rows of $B_{k}$ were isotropic
then we could use Theorem~\ref{thm:Vershynin} directly. However,
unlike $R^{T}Q$ in Lemma~\ref{lem:RQ-isotropic}, this is not the
case for $B_{k}$. While the second moment matrix $\Sigma$ is not
the identity, it does play a special role in the bound of the condition
number since it is the matrix $B_{k}$ from the standard ALS algorithm.
To see this, we take the $\hat{l}-$th row of $B_{k}$, $B_{k}(\hat{l},:)=\left\{ \prod_{i\ne k}\left\langle \mathbf{F}_{i}^{\tilde{l}},\mathbf{R}_{i}^{\hat{l}}\right\rangle \right\} _{\tilde{l}=1,\dots,n}$,
and form the second moment matrix of $X$, 
\begin{eqnarray*}
\Sigma\left(l,l^{\prime}\right) & = & \mathbb{E}\left[\prod_{i\ne k}\left\langle \mathbf{F}_{i}^{l},\mathbf{R}_{i}^{\hat{l}}\right\rangle \left\langle \mathbf{F}_{i}^{l^{\prime}},\mathbf{R}_{i}^{\hat{l}}\right\rangle \right]\\
 & = & \prod_{i\ne k}\mathbb{E}\left[\left(\mathbf{F}_{i}^{l}\right)^{T}\mathbf{R}_{i}^{\hat{l}}\left(\mathbf{R}_{i}^{\hat{l}}\right)^{T}\mathbf{F}_{i}^{l^{\prime}}\right]\\
 & = & \prod_{i\ne k}\left(\mathbf{F}_{i}^{l}\right)^{T}\mathbb{E}\left[\mathbf{R}_{i}^{\hat{l}}\left(\mathbf{R}_{i}^{\hat{l}}\right)^{T}\right]\mathbf{F}_{i}^{l^{\prime}}.
\end{eqnarray*}
Since $\mathbf{R}_{i}^{\hat{l}}$ is a vector composed of either Bernoulli
or standard Gaussian random variables, $\mathbb{E}\left[\mathbf{R}_{i}^{\hat{l}}\left(\mathbf{R}_{i}^{\hat{l}}\right)^{T}\right]=I$.
Therefore, we are left with $\Sigma\left(l,l^{\prime}\right)=\prod_{i\ne k}\left(\mathbf{F}_{i}^{l}\right)^{T}\mathbf{F}_{i}^{l^{\prime}}.$ 

We need to modify Theorem~\ref{thm:Vershynin} for matrices that
have independent, non-isotropic rows. In \cite[Remark 5.40]{VERSHY:2012}
it is noted in the case of a random matrix $A\in\mathbb{R}^{N\times n}$
with non-isotropic rows that we can apply Theorem~\ref{thm:Vershynin}
to $A\Sigma^{-\frac{1}{2}}$ instead of $A$. The matrix $A\Sigma^{-\frac{1}{2}}$
has isotropic rows and, thus, we obtain the following inequality that
holds with probability at least $1-2\exp\left(-ct^{2}\right)$, 
\begin{equation}
\left\Vert \frac{1}{N}A^{T}A-\Sigma\right\Vert \le\max\left(\delta,\delta^{2}\right)\left\Vert \Sigma\right\Vert ,\,\,\,\,\textrm{where}\,\,\delta=C\sqrt{\frac{n}{N}}+\frac{t}{\sqrt{N}},\label{eq:conc-around-sigma}
\end{equation}
and $C=C_{K}$, $c=c_{K}>0$.

To clarify how (\ref{eq:conc-around-sigma}) changes the bounds on
the singular values $\sigma_{\min}\left(A\right)$ and $\sigma_{\max}\left(A\right)$
of a matrix $A$ with non-isotropic rows, we modify Lemma~\ref{lem:Vershynin}. 
\begin{lem}
\label{lem:modified-vershynin-lemma}Consider matrices $B\in\mathbb{R}^{N\times n}$
and $\Sigma^{-\frac{1}{2}}\in\mathbb{R}^{n\times n}$ (non-singular)
that satisfy 
\begin{equation}
1-\delta\le\sigma_{\min}\left(B\Sigma^{-\frac{1}{2}}\right)\le\sigma_{\max}\left(B\Sigma^{-\frac{1}{2}}\right)\le1+\delta,\label{eq:vershynin-SV-bounds-B-sigma}
\end{equation}
for $\delta>0$. Then we have the following bounds on the extreme
singular values of $B$: 
\[
\sigma_{\min}\left(\Sigma^{\frac{1}{2}}\right)\cdot\left(1-\delta\right)\le\sigma_{\min}\left(B\right)\le\sigma_{\max}\left(B\right)\le\sigma_{\max}\left(\Sigma^{\frac{1}{2}}\right)\cdot\left(1+\delta\right).
\]

\end{lem}
\noindent The proof of Lemma~\ref{lem:modified-vershynin-lemma}
is included in Appendix~A. Using Lemma~\ref{lem:modified-vershynin-lemma},
we observe that the bound on the condition number of a matrix $B$
satisfying (\ref{eq:vershynin-SV-bounds-B-sigma}) has the following
form: 
\[
\kappa\left(B\right)\le\frac{\left(1+\delta\right)}{\left(1-\delta\right)}\kappa\left(\Sigma^{\frac{1}{2}}\right).
\]
Using Lemma~\ref{lem:modified-vershynin-lemma}, we prove an extension
of Theorem~\ref{thm:Vershynin} for matrices with non-isotropic rows. 
\begin{thm}
\label{thm:modified-vershinin-theorem}Let $A$ be an $N\times n$
matrix whose rows, $A\left(i,:\right)$, are independent, sub-Gaussian
random vectors in $\mathbb{R}^{n}$. Then for every $t\ge0$, with
probability at least $1-2\exp\left(-ct^{2}\right)$ one has 
\[
\sigma_{\min}\left(\Sigma^{\frac{1}{2}}\right)\cdot\left(\sqrt{N}-C\sqrt{n}-t\right)\le\sigma_{\min}\left(A\right)\le\sigma_{\max}\left(A\right)\le\sigma_{\max}\left(\Sigma^{\frac{1}{2}}\right)\cdot\left(\sqrt{N}+C\sqrt{n}+t\right).
\]
Here $C=C_{K}$, $c=c_{K}>0$, depend only on the sub-Gaussian norm
$K=\underset{i}{\max}\left\Vert A\left(i,:\right)\right\Vert _{\psi_{2}}$
and the norm of $\Sigma^{-\frac{1}{2}}$.\end{thm}
\begin{proof}
We form the second moment matrix $\Sigma$ using rows $A\left(i,:\right)$
and apply Theorem~(\ref{thm:Vershynin}) to the matrix $A\Sigma^{-\frac{1}{2}}$,
which has isotropic rows. Therefore, for every $t\ge0$, with probability
at least $1-2\exp\left(-ct^{2}\right)$, we have 
\begin{equation}
\sqrt{N}-C\sqrt{n}-t\le\sigma_{\min}\left(A\Sigma^{-\frac{1}{2}}\right)\le\sigma_{\max}\left(A\Sigma^{-\frac{1}{2}}\right)\le\sqrt{N}+C\sqrt{n}+t,\label{eq:thrm-A-sigma-bound}
\end{equation}
where $C=C_{\tilde{K}}$, $c=c_{\tilde{K}}>0$, depend only on the
sub-Gaussian norm $\tilde{K}=\underset{i}{\max}\left\Vert \Sigma^{-\frac{1}{2}}A\left(i,:\right)^{T}\right\Vert _{\psi_{2}}$.
Applying Lemma~\ref{lem:modified-vershynin-lemma} to (\ref{eq:thrm-A-sigma-bound})
with $B=A/\sqrt{N}$ and $\delta=C\sqrt{n/N}+t/\sqrt{N}$, results
in the bound 
\[
\sigma_{\min}\left(\Sigma^{\frac{1}{2}}\right)\cdot\left(\sqrt{N}-C\sqrt{n}-t\right)\le\sigma_{\min}\left(A\right)\le\sigma_{\max}\left(A\right)\le\sigma_{\max}\left(\Sigma^{\frac{1}{2}}\right)\cdot\left(\sqrt{N}+C\sqrt{n}+t\right),
\]
with the same probability as (\ref{eq:thrm-A-sigma-bound}).

To move $\Sigma^{-\frac{1}{2}}$ outside the sub-Gaussian norm, we
bound $\tilde{K}$ from above using the sub-Gaussian norm of $A$,
$K=\underset{i}{\max}\left\Vert A\left(i,:\right)\right\Vert _{\psi_{2}},$
\begin{eqnarray*}
\left\Vert \Sigma^{-\frac{1}{2}}A\left(i,:\right)^{T}\right\Vert _{\psi_{2}} & = & \sup_{x\in\mathcal{S}^{n-1}}\left\Vert \left\langle \Sigma^{-\frac{1}{2}}A\left(i,:\right)^{T},x\right\rangle \right\Vert _{\psi_{2}}\\
 & = & \sup_{x\in\mathcal{S}^{n-1}}\frac{\left\Vert \left\langle A\left(i,:\right)^{T},\Sigma^{-\frac{1}{2}}x\right\rangle \right\Vert _{\psi_{2}}}{\left\Vert \Sigma^{-\frac{1}{2}}x\right\Vert _{2}}\left\Vert \Sigma^{-\frac{1}{2}}x\right\Vert _{2}\\
 & \le & \sup_{y\in\mathcal{S}^{n-1}}\left\Vert \left\langle A\left(i,:\right)^{T},y\right\rangle \right\Vert _{\psi_{2}}\,\sup_{x\in\mathcal{S}^{n-1}}\left\Vert \Sigma^{-\frac{1}{2}}x\right\Vert _{2}\\
 & = & \left\Vert A\left(i,:\right)^{T}\right\Vert _{\psi_{2}}\left\Vert \Sigma^{-\frac{1}{2}}\right\Vert ,
\end{eqnarray*}
hence $\tilde{K}\le K\,\left\Vert \Sigma^{-\frac{1}{2}}\right\Vert $.
Using this inequality, we bound the probability in (\ref{eq:vershynin-union-bound})
for the case of Theorem~\ref{thm:Vershynin} applied to $A\Sigma^{-\frac{1}{2}}$.
\begin{eqnarray*}
\mathbb{P}\left\{ \max_{x\in\mathcal{N}}\left|\frac{1}{N}\left\Vert A\Sigma^{-\frac{1}{2}}x\right\Vert _{2}^{2}-1\right|\ge\frac{\epsilon}{2}\right\}  & \le & 9^{n}\cdot2\,\exp\left[-\frac{c_{1}}{\tilde{K}^{4}}\left(C^{2}n+t^{2}\right)\right]\\
 & \le & 9^{n}\cdot2\,\exp\left[-\frac{c_{1}}{K^{4}\,\left\Vert \Sigma^{-\frac{1}{2}}\right\Vert ^{4}}\left(C^{2}n+t^{2}\right)\right]\\
 & \le & 2\,\exp\left(-\frac{c_{1}t^{2}}{K^{4}\,\left\Vert \Sigma^{-\frac{1}{2}}\right\Vert ^{4}}\right).
\end{eqnarray*}
The last step, similar to the proof of Theorem~\ref{thm:Vershynin},
comes from choosing $C$ large enough, for example $C=K^{2}\,\left\Vert \Sigma^{-\frac{1}{2}}\right\Vert ^{2}\sqrt{\ln\left(9\right)/c_{1}}$).

\end{proof}
The combination of Lemma~\ref{lem:modified-vershynin-lemma}, the
fact that $\Sigma$ for \eqref{eq:randomized-ALS-B} is the same matrix
as \eqref{eq:B-matrix-definition} (denoted below as $B_{k}^{ALS}$),
and Theorem~\ref{thm:modified-vershinin-theorem}, leads to our bound
on the condition number of $B_{k}$ in \eqref{eq:randomized-ALS-B}
is, for every $t\ge0$ and with probability at least $1-2\exp\left(-ct^{2}\right)$,
\begin{equation}
\kappa\left(B_{k}\right)\le\frac{1+C\sqrt{r/r^{\prime}}+t/\sqrt{r^{\prime}}}{1-C\sqrt{r/r^{\prime}}-t/\sqrt{r^{\prime}}}\kappa\left(\left(B_{k}^{ALS}\right)^{\frac{1}{2}}\right),\label{eq:Bk-cond-num-bound}
\end{equation}
where the definitions of $C$ and $c$ are the same as in Theorem~\ref{thm:modified-vershinin-theorem}.
\begin{rem}
In both \eqref{eq:bound-on-RtQ} and \eqref{eq:Bk-cond-num-bound}
the ratios $r/r^{\prime}$ and $t/\sqrt{r^{\prime}}$ are present.
As both ratios go to zero, our bound on the condition number of $B_{k}$
goes to $\kappa\left(B_{k}\right)\le\kappa\left(\left(B_{k}^{ALS}\right)^{\frac{1}{2}}\right)$,
and the bound on the condition number of $R^{T}Q$ goes to $\kappa\left(R^{T}Q\right)\le1$.
These properties explain our choice to set $r^{\prime}$ as a constant
multiple of $r$ in the randomized ALS algorithm. As with similar
bounds for randomized matrix algorithms, these bounds are pessimistic.
Hence $r^{\prime}$ does not have to be very large with respect to
$r$ in order to get acceptable results. 

Another reason to choose $r^{\prime}$ as a constant multiple of $r$
is the construction of $R_{k}\in\mathbb{R}^{M\times r^{\prime}}$
in \eqref{eq:Khatri-Rao-defintions} and our choice to use Bernoulli
random numbers. If $r^{\prime}$ is too small, there is the danger
of the matrix $R_{k}$ becoming singular due to the size of $M$ and
repeated rows. Choosing the constant multiple that defines $r^{\prime}$
large enough helps mitigate this problem. 
\end{rem}

\section{Examples\label{sec:Examples}}

\subsection{Sine function}

Our first test of the randomized ALS algorithm is to reduce a CTD
generated from samples of the multivariate function $\sin\left(z_{1}+\dots+z_{d}\right)$.
This reduction problem was studied in \cite{BEY-MOH:2005}, where
the output of the standard ALS algorithm suggested a new trigonometric
identity yielding a rank $d$ separated representation of $\sin\left(z_{1}+\dots+z_{d}\right)$.
As input, we use standard trigonometric identities to produce a rank
$2^{d-1}$ initial CTD.

We ran $500$ tests using both standard ALS and the new randomized
algorithm to reduce the separation rank of a CTD of samples of $\sin\left(z_{1}+\dots+z_{d}\right)$.
The tests differed in that each one had a different random initial
guess with separation rank $r_{\mathbf{F}}=1$. In this example we
chose $d=5$ and sampled each variable $z_{i}$, $i=1,\dots,d$, with
$M=64$ equispaced samples in the interval $\left[0,2\pi\right]$.
Our input CTD for both algorithms was rank $16$ and was generated
via a standard trigonometric identity. The reduction tolerance for
both algorithms was set to $\epsilon=10^{-5}$, and the maximum number
of iterations per rank, i.e. $max\_iter$ in Algorithm~\ref{alg:Alternating-least-squares}
and Algorithm~\ref{alg:Randomized-alternating-least-squares}, was
set to $1000$. For tests involving the standard ALS algorithm we
used a stuck tolerance of $\delta=10^{-8}$. To test the randomized
ALS algorithm we used $B_{k}$ matrices of size $\left(25\, r_{\mathbf{F}}\right)\times r_{\mathbf{F}}$
and set $max\_tries$ in Algorithm~\ref{alg:Randomized-alternating-least-squares}
to $50$.

According to Lemma 2.4 in \cite{BEY-MOH:2005}, there exists exact
rank $5$ separated representations of $\sin\left(z_{1}+\dots+z_{d}\right)$.
Using $\epsilon=10^{-5}$ for our reduction tolerance, we were able
to find rank $5$ approximations with both standard ALS and our randomized
ALS whose relative errors were less than the requested $\epsilon$
(for a histogram of residuals of the tests, see Figure~\ref{fig:sine-residuals}). 

Due to the random initial guess $\mathbf{F}$ and our choices of parameters
(in particular the stuck tolerance and $max\_tries$) both algorithms
had a small number of runs that did not find rank $5$ approximations
with the requested tolerance $\epsilon$. The randomized ALS algorithm
produced fewer of these outcomes than standard ALS.

Large differences in maximum condition number (of ALS solves) are
illustrated in Figure~\ref{fig:sine-cond-num}, where we compare
tests of the standard and randomized ALS algorithms. We observe that
the maximum condition numbers produced by the randomized ALS algorithm
are much smaller those from the standard ALS algorithm. This is consistent
with our theory. 

Furthermore, as shown in Figure~\ref{fig:sine_iter}, the number
of iterations required for randomized ALS to converge was smaller
than the number required by standard ALS. It is important to remember
that the number of iterations required by the standard ALS algorithm
to reduce a CTD can be optimized by adjusting the tolerance, stuck
tolerance, and maximum number of iterations per rank. In these experiments
we chose the stuck tolerance and maximum number of iterations to reduce
the number of tests of the standard ALS algorithm that did not meet
the requested tolerance $\epsilon$.

\begin{figure}
\begin{centering}
\includegraphics{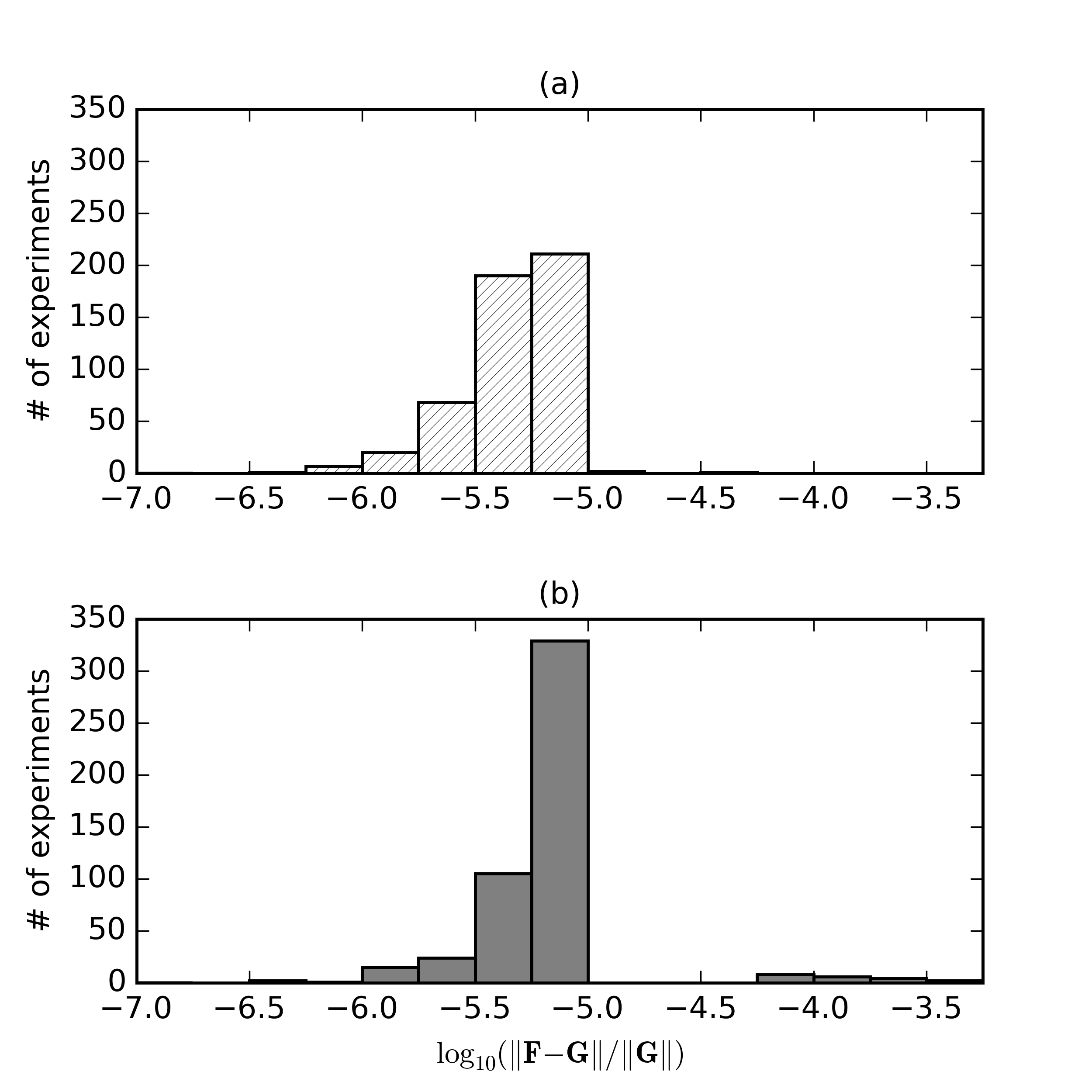} 
\par\end{centering}

\protect\protect\protect\protect\caption{\label{fig:sine-residuals}Histograms displaying ALS reduction residuals,
in $\log_{10}$ scale, for reducing the length of a CTD of samples
of $\sin\left(z_{1}+\dots z_{5}\right)$. The experiments shown in
(a) used randomized ALS, whereas the experiments shown in (b) used
standard ALS. We note that both algorithms produced a small number
of results with approximation errors worse than the requested tolerance.
However, the randomized ALS method produced fewer results that did
not meet our requested tolerance.}
\end{figure}

\begin{figure}
\begin{centering}
\includegraphics{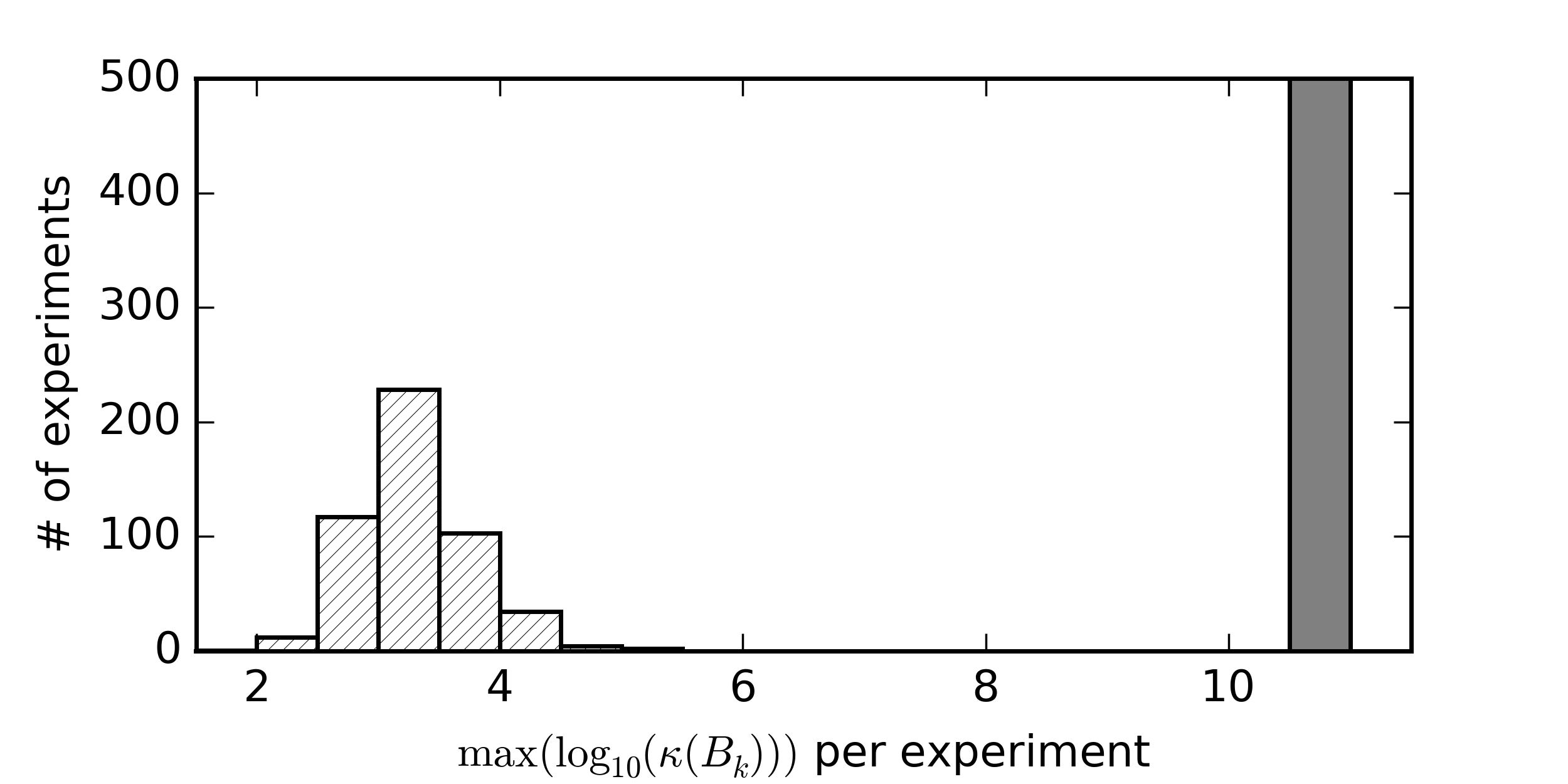} 
\par\end{centering}

\protect\protect\protect\protect\caption{\label{fig:sine-cond-num}Histogram showing the maximum condition
numbers from our experiments reducing the length of a CTD of samples
of $\sin\left(z_{1}+\dots z_{5}\right)$. The condition numbers are
shown in $\log_{10}$ scale; the solid gray pattern represents condition
numbers from standard ALS, while the hatch pattern represents condition
numbers from the randomized ALS algorithm.}
\end{figure}

\begin{figure}
\begin{centering}
\includegraphics{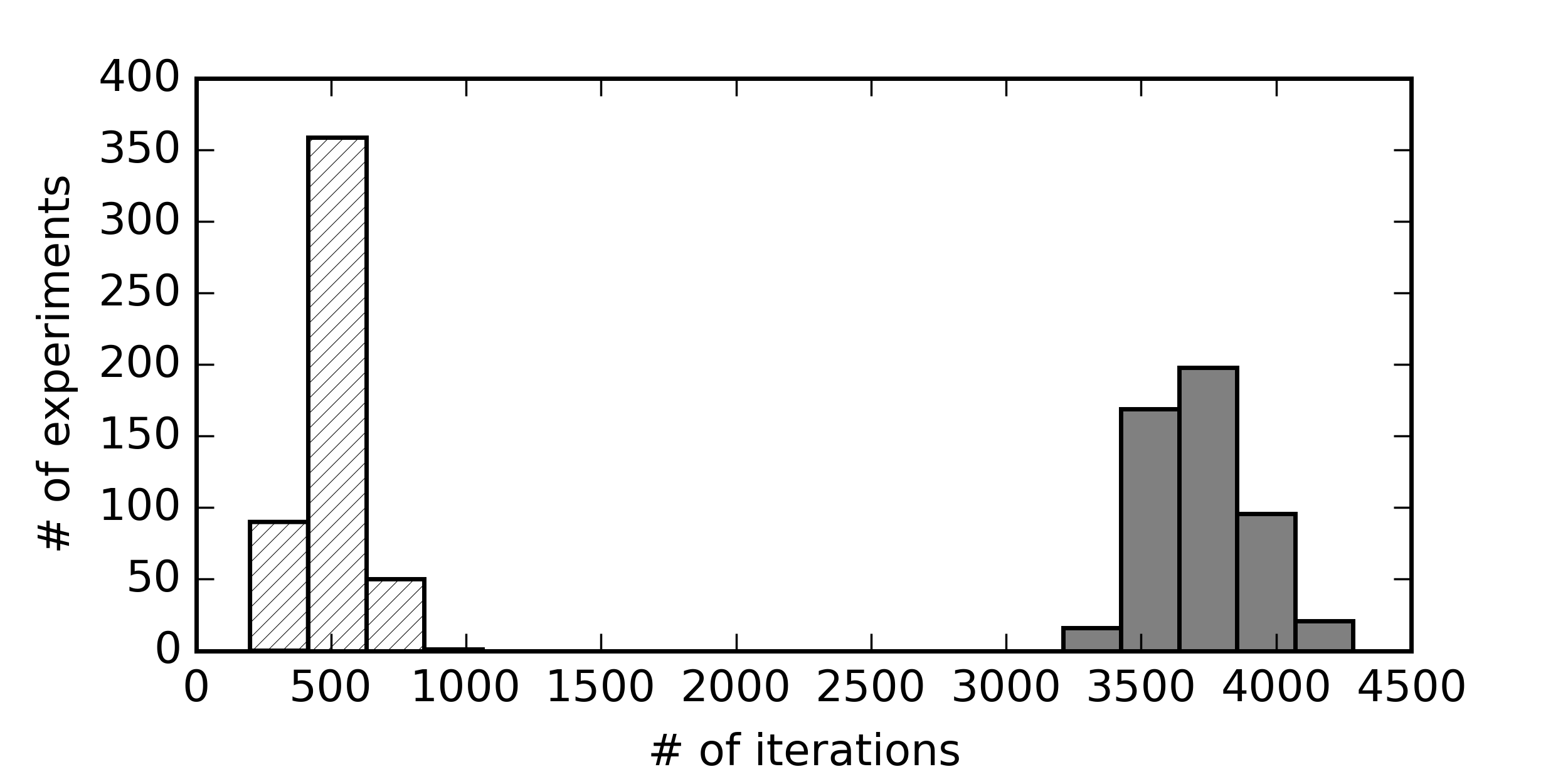} 
\par\end{centering}

\protect\protect\protect\protect\caption{\label{fig:sine_iter}Histogram showing the number of iterations required
by randomized ALS (hatch pattern) and the standard ALS algorithm (gray
pattern) to reduce the length of a CTD of samples of $\sin\left(z_{1}+\dots z_{5}\right)$.}
\end{figure}

\subsection{A manufactured tensor}

Our next test is to compare the performance of the standard and randomized
ALS algorithms on a manufactured random tensor example. To construct
this example we generate factors by drawing $M=128$ random samples
from the standard Gaussian distribution. We chose $d=10$ and set
the separation rank of the input tensor to $r=50$. Then we normalized
the factors and set the s-values of the tensor equal to $s_{l}=e^{-l}$,
$l=0,\dots,r-1$, where $r$ was predetermined such that $s_{end}$
is small.

Similar to the sine example, we ran $500$ experiments and requested
an accuracy of $\epsilon=10^{-4}$ from both algorithms. The maximum
number of iterations for both algorithms was set to $1000$, while
the stuck tolerance for the standard ALS algorithm was set to $10^{-6}$.
We used the following parameters for the randomized ALS algorithm:
the $B_{k}$ matrices were of size $\left(25\, r_{\mathbf{F}}\right)\times r_{\mathbf{F}}$,
and the repetition parameter, $max\_tries$ in Algorithm~\ref{alg:Randomized-alternating-least-squares},
was set to $50$. We started all tests from randomized guesses with
rank $\mbox{r}_{\mathbf{F}}=9$ . This value was chosen because in
all previous runs the reduced separation rank never fell below $r_{\mathbf{F}}=10$.
Such an experiment allows us to compare how the algorithms perform
when the initial approximation has rank greater than one.

We show in Figure~\ref{fig:rand-output-ranks} the output separation
ranks from $500$ tests of both the randomized and standard ALS algorithms.
The CTD outputs from randomized ALS had, on average, lower separation
ranks than those from standard ALS. Furthermore, as seen in Figure~\ref{fig:rand-output-ranks},
some of the output CTDs from the standard ALS algorithm had separation
rank above $40$. In these instances, standard ALS failed to reduce
the separation rank of the input CTD because simple truncation to
$r_{\mathbf{F}}=35$ would have given double precision. These failures
did not occur with the randomized ALS algorithm. We can also see the
contrast in performance in Figure~\ref{fig:rand-residuals}: all
tests of the randomized ALS algorithm produced CTDs with reduced separation
rank whose relative reduction errors were less than the accuracy $\epsilon$.
Also, in Figure~\ref{fig:rand-residuals}, we observe instances where
the standard ALS algorithm failed to output a reduced separation rank
CTD with relative error less than $\epsilon$.

There was a significant difference in the maximum condition numbers
of matrices used in the two algorithms. In Figure~\ref{fig:rand_max_cond},
we see that matrices produced by standard ALS had much larger condition
numbers (by a factor of roughly $10^{10}$) than their counterparts
in the randomized ALS algorithm. Such large condition numbers may
explain the failures of the standard ALS algorithm to output reduced
separation rank CTDs with relative errors less than $\epsilon$.

From Figure~\ref{fig:rand-iterations}, we see that in many cases
standard ALS required fewer iterations than the randomized ALS algorithm
to converge. However, relative to randomized ALS, standard ALS required
a large number of iterations for a considerable fraction of the experiments.

\begin{figure}
\centering{}\includegraphics{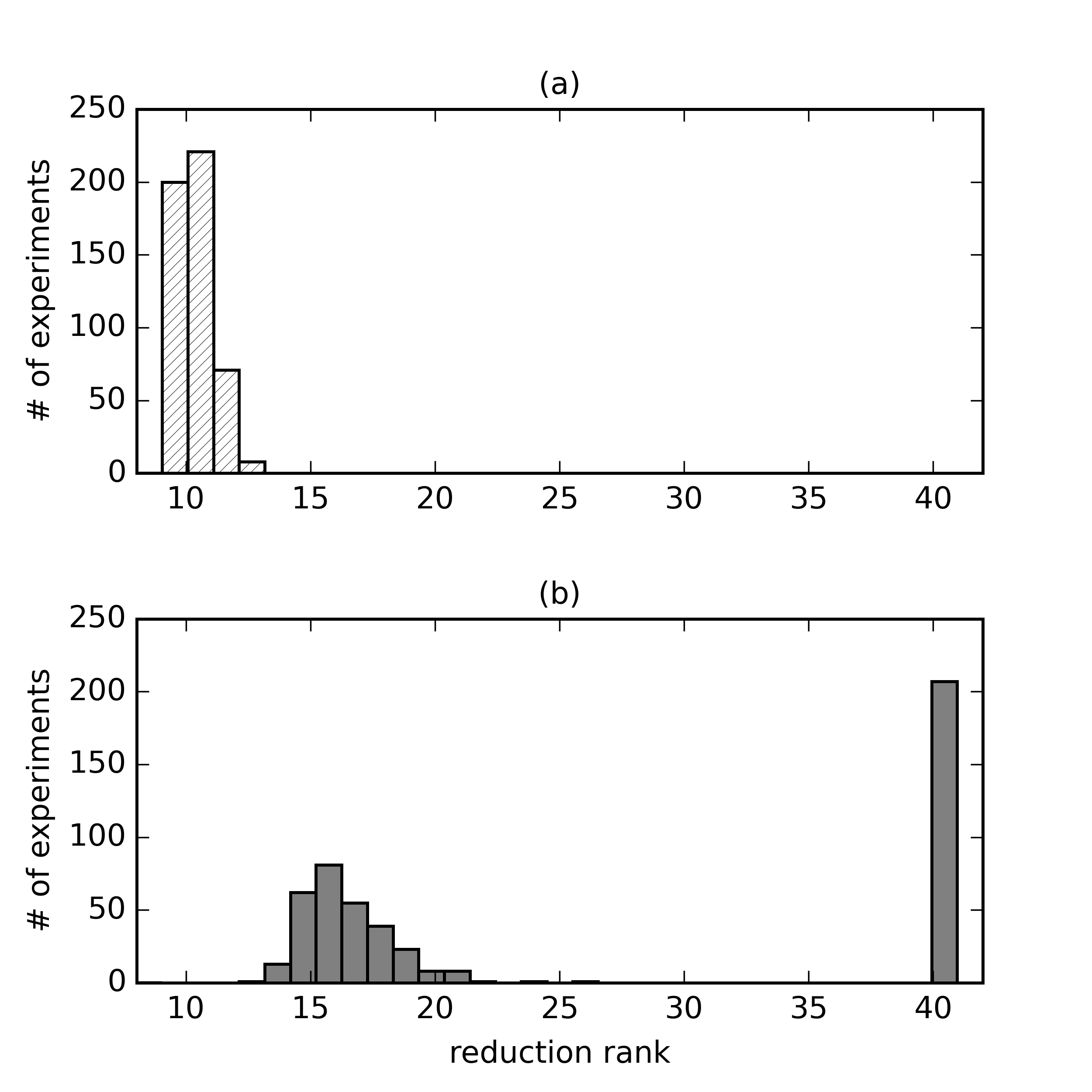}\protect\protect\protect\protect\caption{\label{fig:rand-output-ranks}Histograms showing output ranks from
experiments in reducing the length of CTDs. (a) shows ranks of CTDs
output by the randomized ALS algorithm. (b) shows ranks of CTDs output
by the standard ALS algorithm. The CTDs output by the randomized ALS
method typically have a smaller separation rank. In many examples
the standard ALS algorithm required 40 terms, i.e. it failed since
truncation of the input tensor to $r_{\mathbf{F}}=35$ should give
double precision. }
\end{figure}

\begin{figure}
\begin{centering}
\includegraphics{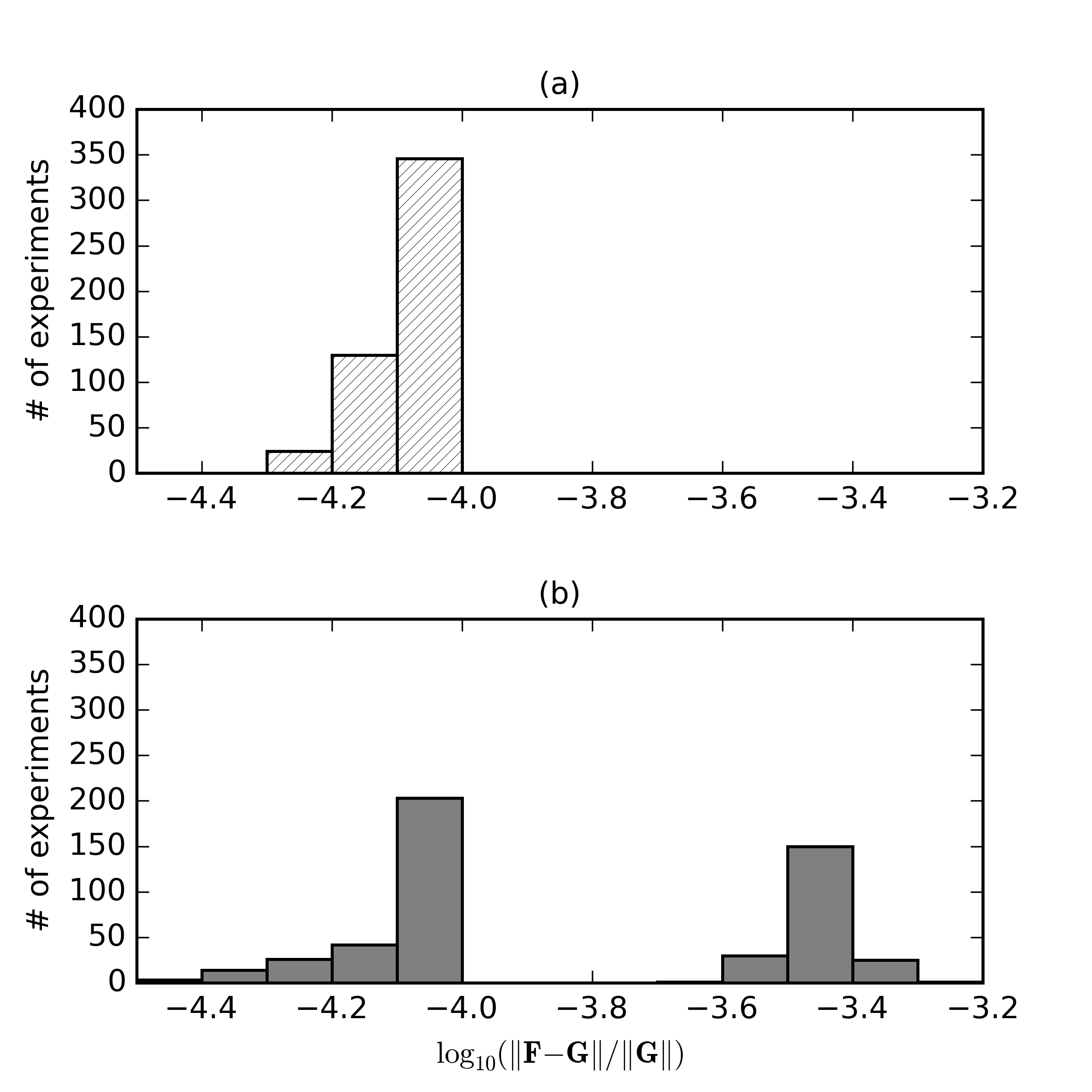} 
\par\end{centering}

\protect\protect\protect\protect\caption{\label{fig:rand-residuals}Histograms displaying ALS reduction errors,
in $\log_{10}$ scale, for reduced-rank CTDs of the random tensor
example. (a) shows that in our 500 tests, the randomized ALS method
always produced a result that met the required tolerance. (b) shows
how the standard ALS algorithm fared with the same problem. Note that
the standard ALS algorithm failed to reach the requested tolerance
in a significant number of tests.}
\end{figure}

\begin{figure}
\centering{}\includegraphics{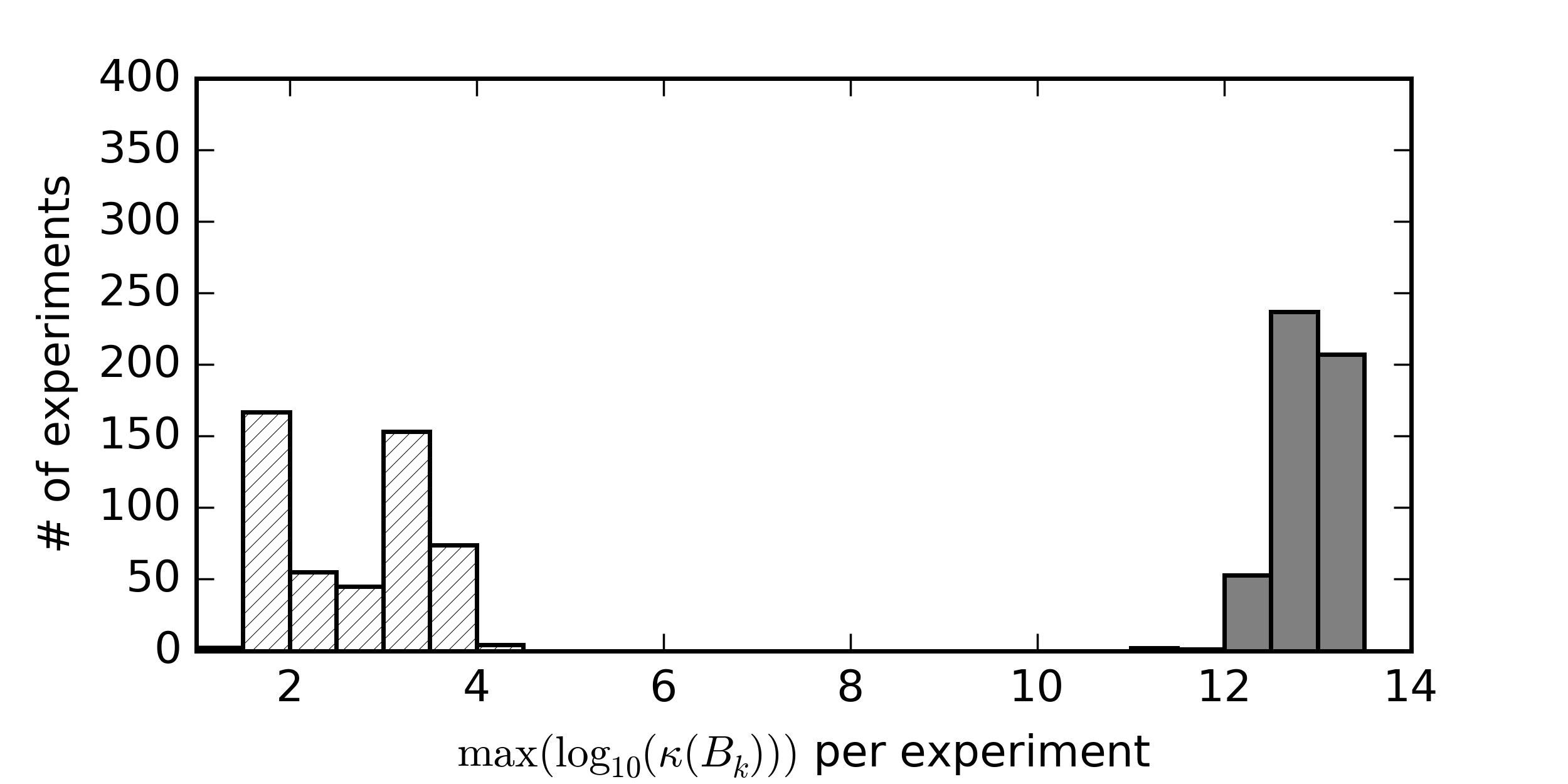}\protect\protect\protect\protect\caption{\label{fig:rand_max_cond}Histogram showing the maximum condition
numbers from experiments in reducing the length of CTDs of the random
tensor example. The condition numbers of $B_{k}$ are shown in $\log_{10}$
scale; solid gray represents condition numbers from standard ALS while
the hatch pattern represents condition numbers from the randomized
ALS algorithm. Similar to the sine example, the condition numbers
from randomized ALS are much smaller than those from the standard
ALS algorithm}
\end{figure}

\begin{figure}
\centering{}\includegraphics{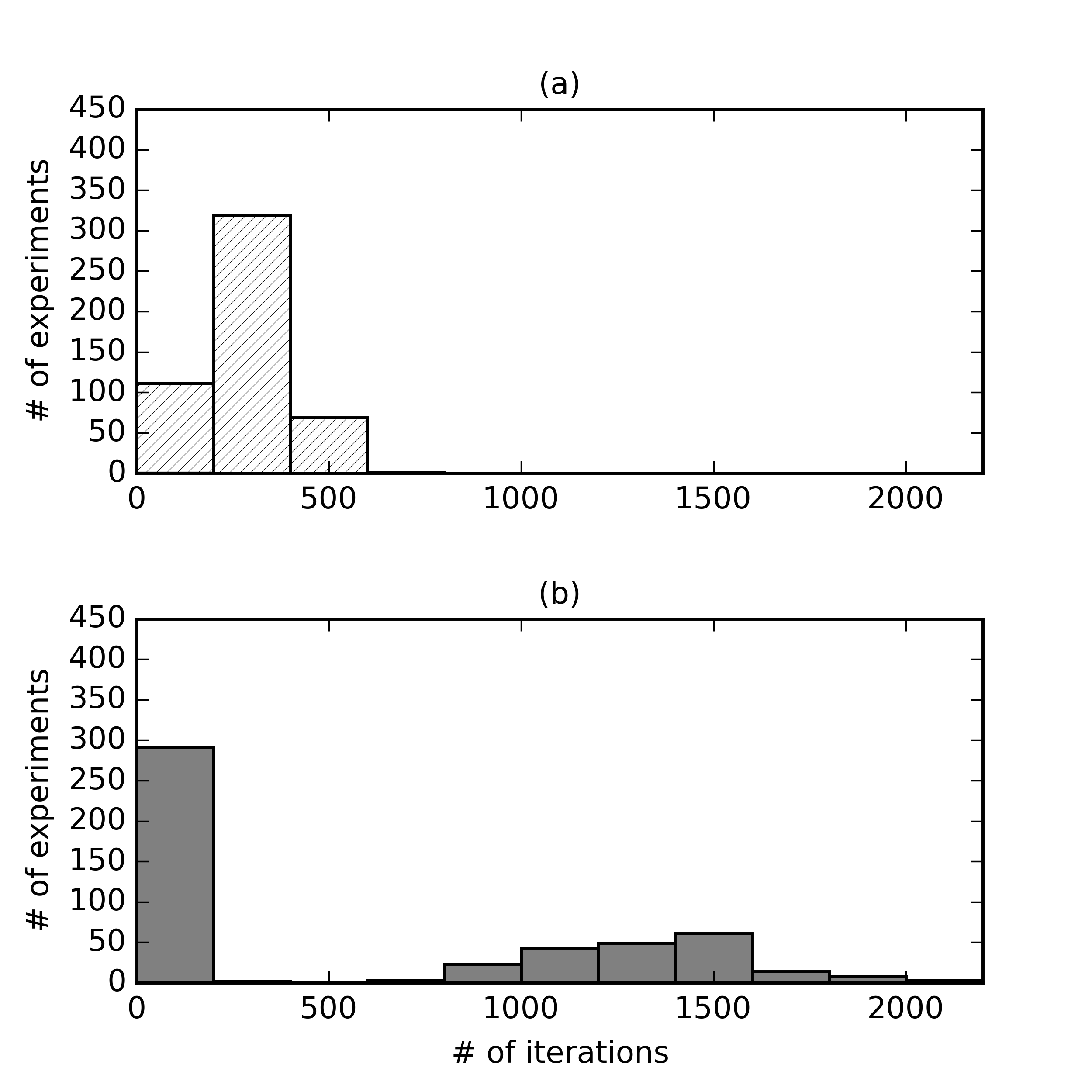}\protect\protect\protect\protect\caption{\label{fig:rand-iterations}Histograms showing iterations required
to produce reduced-length CTDs for the random tensor example. (a)
shows iterations required by randomized ALS, while (b) shows the iterations
required by the standard ALS algorithm. As seen in (b), many examples
using the standard ALS algorithm required large numbers of iterations
to output CTDs.}
\end{figure}

\subsection{Elliptic PDE with random coefficient\label{sub:Elliptic-PDE-rand-coeff}}

\label{subsec:Elliptic} As the key application of the randomized
ALS algorithm, we consider the separated representation of the solution
$u(\mathbf{x},\mathbf{z})$ to the linear elliptic PDE 
\begin{eqnarray}
-\nabla\cdot(a(\mathbf{x},\mathbf{z})\nabla u(\mathbf{x},\mathbf{z})) & = & 1,\quad\mathbf{x}\in\mathcal{D},\label{eqn:twod_elliptic}\\
u(\mathbf{x},\mathbf{z}) & = & 0,\quad\mathbf{x}\in\partial\mathcal{D},\nonumber 
\end{eqnarray}
defined on the unit square $\mathcal{D}=(0,1)\times(0,1)$ with boundary
$\partial\mathcal{D}$. The diffusion coefficient $a(\mathbf{x},\mathbf{z})$
is considered random and is modeled by 
\begin{align}
a(\mathbf{x},\mathbf{z})=a_{0}+\sigma_{a}\sum_{k=1}^{d}\sqrt{\zeta_{k}}\varphi_{k}(\mathbf{x})z_{k},\label{eqn:gaussian_field}
\end{align}
where $\mathbf{z}=(z_{1},\dots,z_{d})$ and the random variables $z_{k}$
are independent and uniformly distributed over the interval $[-1,1]$,
and we choose $a_{0}=0.1$, $\sigma_{a}=0.01$, and $d=5$. In (\ref{eqn:gaussian_field}),
$\{\zeta_{k}\}_{k=1}^{d}$ are the $d$ largest eigenvalues associated
with $\{\varphi_{k}\}_{k=1}^{d}$, the $L_{2}(\mathcal{D})$-orthonormalized
eigenfunctions of the exponential covariance function 
\begin{equation}
C_{aa}(\mathbf{x}_{1},\mathbf{x_{2}})=\exp{\left(-\frac{\Vert\mathbf{x}_{1}-\mathbf{x}_{2}\Vert_{1}}{l_{c}}\right)},\label{eq:GaussCov}
\end{equation}
where $l_{c}$ denotes the correlation length, here set to $l_{c}=2/3$.
Given the choices of parameters $a_{0}$, $\sigma_{a}$, $d$, and
$l_{c}$, the model in (\ref{eqn:gaussian_field}) leads to strictly
positive realizations of $a(\mathbf{x},\mathbf{z})$.

We discretize (\ref{eqn:twod_elliptic}) in the spatial domain $\mathcal{D}$
via triangular finite elements of size $h=1/32$. This discretization,
along with the affine representation of $a(\mathbf{x},\mathbf{z})$
in $z_{k}$, yields the random linear system of equations 
\begin{equation}
\left(K_{0}+\sum_{k=1}^{d}K_{k}z_{k}\right)\mathbf{u}(\mathbf{z})=\mathbf{f},\label{eqn:ex-1-semi-disc}
\end{equation}
for the approximate vector of nodal solutions $\mathbf{u}(\mathbf{z})\in\mathbb{R}^{N}$.
The sparse matrices $K_{0}$ and $K_{k}$ are obtained from the finite
element discretization of the differential operator in (\ref{eqn:twod_elliptic})
assuming $a(\mathbf{x},\mathbf{z})$ is replaced by $\bar{a}$ and
$\sigma_{a}\sqrt{\zeta_{k}}\phi_{k}(\mathbf{x})$, respectively.

To fully discretize (\ref{eqn:ex-1-semi-disc}), we consider the representation
of $\mathbf{u}(\mathbf{z})$ at a tensor-product grid $\left\{ \left(z_{1}(j_{1}),\dots,z_{d}(j_{d})\right):j_{k}=1,\dots,M_{k}\right\} $
where, for each $k$, the grid points $z_{k}(j_{k})$ are selected
to be the Gauss-Legendre abscissas. In our numerical experiments,
we used the same number of abscissas $M_{k}=M=8$ for all $k=1,\dots,d$.
The discrete representation of (\ref{eqn:ex-1-semi-disc}) is then
given by the tensor system of equations 
\begin{equation}
\mathbb{K}\mathbf{U}=\mathbf{F},\label{eqn:tensor_equation}
\end{equation}
where the linear operation $\mathbb{K}\mathbf{U}$ is defined as 
\[
\mathbb{K}\mathbf{U}=\sum_{\hat{l}=0}^{d}\sum_{\tilde{l}=1}^{r_{\mathbf{U}}}s_{\tilde{l}}^{\mathbf{U}}\left(\mathbb{K}_{0}^{\hat{l}}\mathbf{U}_{0}^{\tilde{l}}\right)\circ\left(\mathbb{K}_{1}^{\hat{l}}\mathbf{U}_{1}^{\tilde{l}}\right)\circ\dots\circ\left(\mathbb{K}_{d}^{\hat{l}}\mathbf{U}_{d}^{\tilde{l}}\right),
\]
\[
\mathbb{K}_{0}^{\hat{l}}=K_{\hat{l}},
\]
and for $k=1,\dots,d$, 
\[
\mathbb{K}_{k}^{\hat{l}}=\begin{cases}
D & \hat{l}=k\\
I_{M} & \hat{l}\ne k,
\end{cases}
\]
 where
\[
D=\left[\begin{array}{ccc}
z_{k}\left(1\right) &  & 0\\
 & \ddots\\
0 &  & z_{k}\left(M\right)
\end{array}\right],
\]
and $I_{M}$ is the $M\times M$ identity matrix. The tensor $\mathbf{F}$
in \eqref{eqn:tensor_equation} is defined as
\[
\mathbf{F}=\mathbf{f}\circ{\bf 1}_{M}\circ\dots\circ{\bf 1}_{M},
\]
where ${\bf 1}_{M}$ is an $M$-vector of ones. We seek to approximate
$\mathbf{U}$ in (\ref{eqn:tensor_equation}) with a CTD, 
\begin{equation}
\mathbf{U}=\sum_{\tilde{l}=1}^{r_{\mathbf{U}}}s_{\tilde{l}}^{\mathbf{U}}\mathbf{U}_{0}^{\tilde{l}}\circ\mathbf{U}_{1}^{\tilde{l}}\circ\dots\circ\mathbf{U}_{d}^{\tilde{l}},\label{eqn:tensor_sol}
\end{equation}
where the separation rank $r_{\mathbf{U}}$ will be determined by
a target accuracy. In \eqref{eqn:tensor_sol} $\mathbf{U}_{0}^{\tilde{l}}\in\mathbb{R}^{N}$
and $\mathbf{U}_{k}^{\tilde{l}}\in\mathbb{R}^{M}$, $k=1,\dots,d$.
To solve (\ref{eqn:tensor_equation}), we use a fixed point iteration
similar to those used for solving matrix equations and recently employed
to solve tensor equations in \cite{KHO-SCH:2011}. In detail, the
iteration starts with an initial tensor $\mathbf{U}$ of the form
in (\ref{eqn:tensor_sol}). At each iteration $i$, $\mathbf{U}$
is updated according to %

\[
\mathbf{U}_{i+1}=\left(\mathbb{I}-\mathbb{K}\right)\mathbf{U}_{i}+\mathbf{F},
\]
while requiring $\left\Vert \mathbb{I}-\mathbb{K}\right\Vert <1$.
To assure this requirement is satisfied we solve 
\begin{equation}
\mathbf{U}_{i+1}=c\left(\mathbf{F}-\mathbb{K}\mathbf{U}_{i}\right)+\mathbf{U}_{i},\label{eqn:fixed-pt-it}
\end{equation}
where $c$ is chosen such that $\left\Vert \mathbb{I}-c\,\mathbb{K}\right\Vert <1$.
We compute the operator norm $\left\Vert \mathbb{I}-\mathbb{K}\right\Vert $
via power method; see, e.g., \cite{BEY-MOH:2005,BI-BE-BE:2015}. 

One aspect of applying such an iteration to a CTD is an increase in
the output separation rank. For example, if we take a tensor $\mathbf{U}$
of separation rank $r_{\mathbf{U}}$ and use it as input for (\ref{eqn:fixed-pt-it}),
one iteration would increase the rank to $r_{\mathbf{F}}+\left(d+2\right)\, r_{\mathbf{U}}$.
Therefore we require a reduction algorithm to decrease the separation
rank as we iterate. This is where either the standard or randomized
ALS algorithm is required: to truncate the separated representation
after we have run an iteration. Both ALS methods work with a user-supplied
truncation accuracy $\epsilon$, so we denote the reduction operator
as $\tau_{\epsilon}$. Including this operator into our iteration,
we have 
\begin{equation}
\mathbf{U}_{i+1}=\tau_{\epsilon}\left(c\left(\mathbf{F}-\mathbb{K}\mathbf{U}_{i}\right)+\mathbf{U}_{i}\right).\label{eq:fixed-pt-it-with-reduction}
\end{equation}
Pseudocode for our fixed point algorithm is shown in Algorithm~\ref{alg:recursive_PDE_solver}

\begin{algorithm}[h]
\Input{$\epsilon>0,\;\mu>0,\;\textrm{operator}\;\mathbb{K},\;\mathbf{F},\; c,\; max\_iter,\; max\_rank,$
$\delta>0$ (for standard ALS), $max\_tries$ (for randomized ALS)}

\Initialize $r_{\mathbf{U}}=1$ tensor $\mathbf{U}_{0}=\mathbf{U}_{1}^{1}\circ\dots\circ\mathbf{U}_{d}^{1}$
with either randomly generated 

~~~~~~~~~~~~~~$\mathbf{U}_{k}^{1}$ or $\mathbf{U}_{k}^{1}$
generated from the solution of the nominal equations.

$\mathbf{D}_{0}=\mathbf{F}-\mathbb{K}\mathbf{U}_{0}$

$res=\left\Vert \mathbf{D}_{0}\right\Vert /\left\Vert \mathbf{F}\right\Vert $

$iter=0$

\While{$res>\mu$} {

$iter=iter+1$

$\mathbf{U}_{iter}=c\,\mathbf{D}_{iter-1}+\mathbf{U}_{iter-1}$

$\mathbf{U}_{iter}=\tau_{\epsilon}\left(\mathbf{U}_{iter}\right)$

$\mathbf{D}_{iter}=\mathbf{F}-\mathbb{K}\mathbf{U}_{iter}$

$res=\left\Vert \mathbf{D}_{iter}\right\Vert /\left\Vert \mathbf{F}\right\Vert $

}

\protect

\Return{ $\mathbf{U}_{iter}$}

~

\protect\caption{Fixed point iteration algorithm for solving \eqref{eq:fixed-pt-it-with-reduction}\label{alg:recursive_PDE_solver}}
\end{algorithm}

\begin{rem}
In this example, the separation rank of $\mathbb{K}$ is directly
related to the problem dimension $d$, i.e. $r_{\mathbb{K}}=d+1$,
which is a consequence of using a Karhunen-Loeve-type expansion for
finite-dimensional noise representation of $a(\mathbf{x},\mathbf{z})$.
This will increase the computational cost of the algorithm to more
than linear with respect to $d$, e.g. quadratic in $d$ when an iterative
solver is used and $N\gg M$. Alternatively, one can obtain the finite-dimensional
noise representation of $a(\mathbf{x},\mathbf{z})$ by applying the
separated rank reduction technique of this study on the stochastic
differential operator itself to possibly achieve $r_{\mathbb{K}}<d$.
The interested reader is referred to \cite{BEY-MOH:2002,BEY-MOH:2005}
for more details. 
\end{rem}
First, we examine the convergence of the iterative algorithm given
a fixed ALS reduction tolerance in Figure~\ref{fig:fixed-tolerance-iterative-solve}.
The randomized ALS method converges to more accurate solutions in
all of these tests (see Table~\ref{tab:Table-constant-tol}). However,
the ranks of the randomized ALS solutions are larger than the ranks
required for solutions produced by the standard ALS algorithm.

\begin{figure}
\centering{}\includegraphics{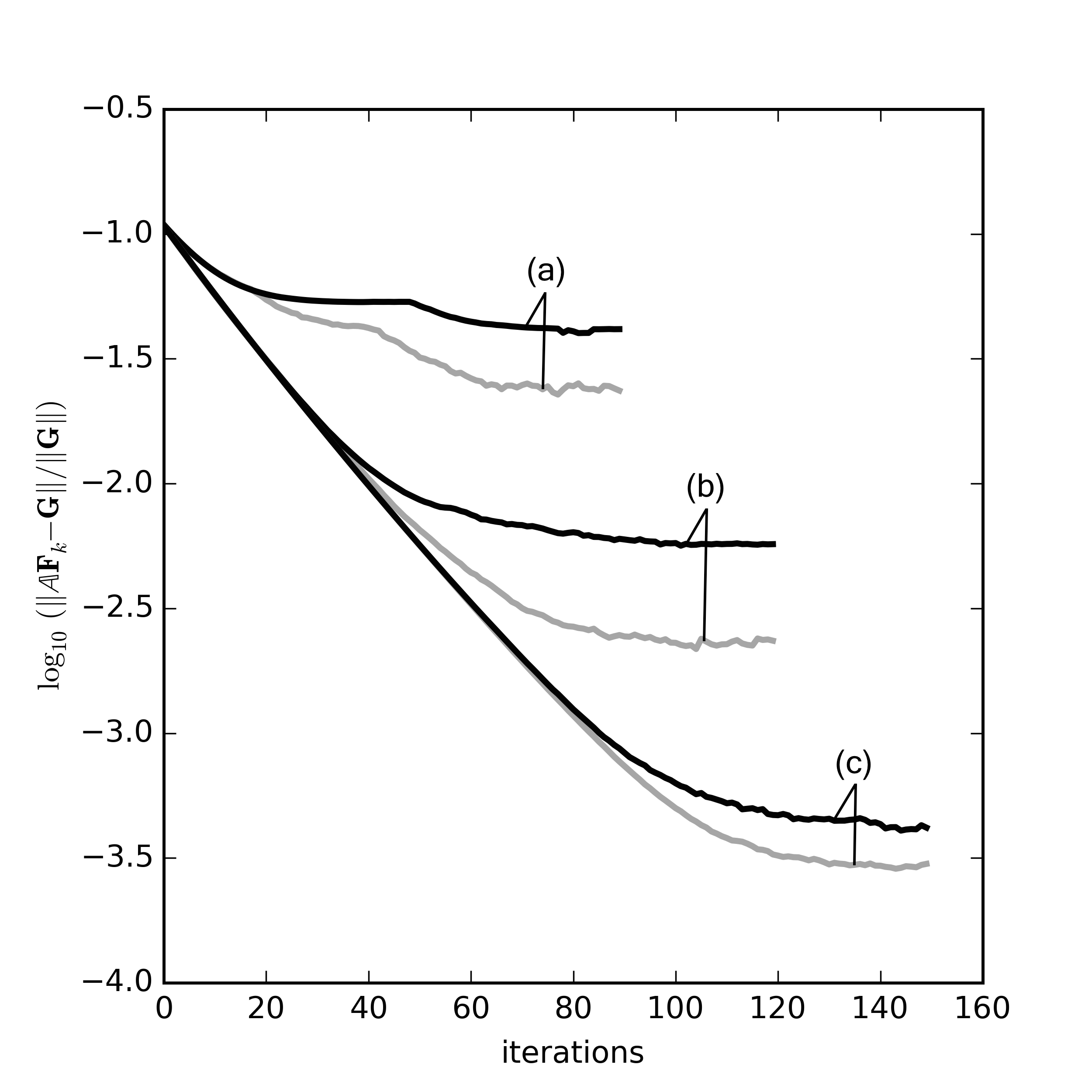}\protect\protect\protect\protect\caption{\label{fig:fixed-tolerance-iterative-solve}Residual error vs. iteration
of results from linear solvers. The black lines represent linear solve
residuals where standard ALS was used for reduction, while the gray
lines represent linear solve residuals where randomized ALS was used
for reduction. In the three examples shown above the ALS tolerances,
for both standard and randomized ALS, were set to $1\times10^{-3}$
for curves labeled (a), $1\times10^{-4}$ for curves labeled (b),
and $1\times10^{-5}$ for curves labeled (c). }
\end{figure}

In Figure~\ref{fig:fixed-rank-iterative-solve}, we observe different
behavior in the relative residuals using fixed ranks instead of fixed
accuracies. For these experiments the ALS-based linear solve using
the standard algorithm out-performs the randomized version, except
in the rank $30$ case (see Table~\ref{tab:Table-constant-rank}).
In this case, the standard ALS algorithm has issues reaching the requested
ALS reduction tolerance, thus leading to convergence problems in the
iterative linear solve. The randomized ALS algorithm does not have
the same difficulty with the rank $r=30$ example. The difference
in performance between standard and randomized ALS for this example
corresponds to a significant difference between the maximum condition
numbers of $B_{k}$. For the $r=30$ case, the maximum condition number
of $B_{k}$ matrices generated by randomized ALS was $3.94\times10^{7}$,
whereas the maximum condition number of $B_{k}$ matrices generated
by standard ALS was $3.00\times10^{13}$.

\begin{figure}
\centering{}\includegraphics{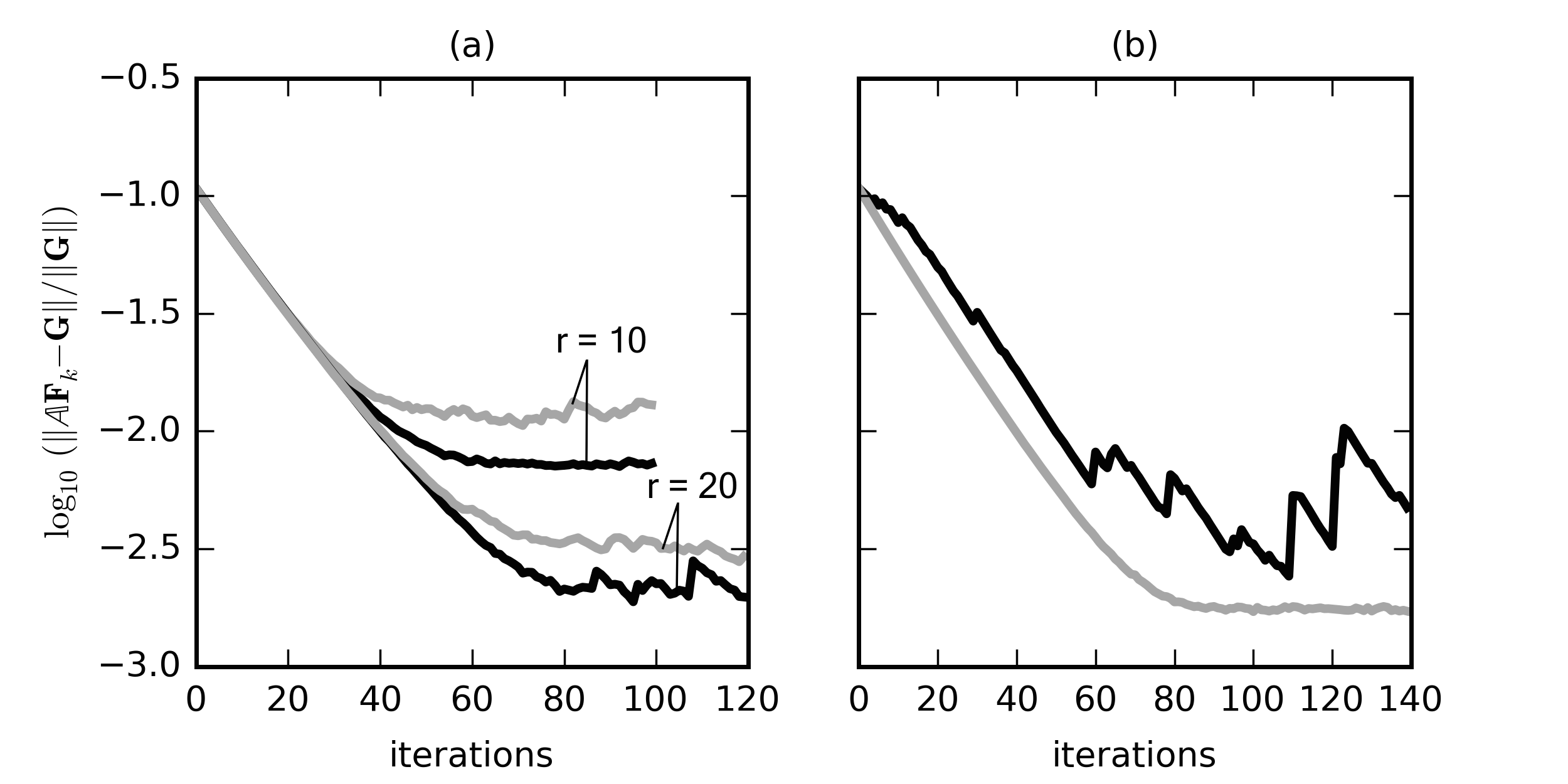}\protect\protect\protect\protect\caption{\label{fig:fixed-rank-iterative-solve}Plots showing relative residuals
of linear solves vs. fixed point iteration number. (a) two linear
solves are shown here: solves for fixed ranks $r=10$, and $r=20$.
Gray lines are residuals corresponding to reductions with randomized
ALS and black lines correspond to reductions with the standard ALS
algorithm. (b) One experiment with $r=30$ and the same color scheme
as in (a).}
\end{figure}

\begin{table}
\global\long\def\arraystretch{1.5}

\begin{centering}
\begin{tabular}{|c|c|c|c|c|c|}
\hline 
ALS type  & ALS tol  & max $\kappa\left(B_{k}\right)$  & max rank  & rank  & residual\tabularnewline
\hline 
\hline 
standard & $1\times10^{-3}$  & $5.35\times10^{1}$  & $5$  & $4$  & $4.16\times10^{-2}$\tabularnewline
\hline 
 & $1\times10^{-4}$  & $5.29\times10^{5}$  & $13$  & $11$  & $5.72\times10^{-3}$\tabularnewline
\hline 
 & $1\times10^{-5}$  & $1.07\times10^{9}$  & $37$  & $34$  & $4.18\times10^{-4}$\tabularnewline
\hline 
randomized  & $1\times10^{-3}$  & $2.59\times10^{2}$  & $7$  & $6$  & $2.36\times10^{-2}$\tabularnewline
\hline 
 & $1\times10^{-4}$  & $3.59\times10^{3}$  & $22$  & $19$  & $2.35\times10^{-3}$\tabularnewline
\hline 
 & $1\times10^{-5}$  & $2.72\times10^{4}$  & $57$  & $54$  & $3.00\times10^{-4}$\tabularnewline
\hline 
\end{tabular}
\par\end{centering}

~%

~

\protect\protect\protect\protect\caption{Table containing ranks, maximum condition numbers, and final relative
residual errors of experiments with fixed ALS tolerance.\label{tab:Table-constant-tol}}
\end{table}

\begin{table}
\global\long\def\arraystretch{1.5}

\begin{centering}
\begin{tabular}{|c|c|c|c|c|}
\hline 
ALS type  & ALS tol  & max $\kappa\left(B_{k}\right)$  & rank  & residual\tabularnewline
\hline 
\hline 
standard  & $1\times10^{-5}$  & $9.45\times10^{11}$  & $10$  & $7.29\times10^{-3}$\tabularnewline
\hline 
 & $5\times10^{-6}$  & $1.27\times10^{13}$  & $20$  & $1.97\times10^{-3}$\tabularnewline
\hline 
 & $1\times10^{-6}$  & $3.00\times10^{13}$  & $30$  & $4.73\times10^{-3}$\tabularnewline
\hline 
randomized  & $1\times10^{-5}$  & $9.39\times10^{5}$  & $10$  & $1.30\times10^{-2}$\tabularnewline
\hline 
 & $5\times10^{-6}$  & $4.12\times10^{6}$  & $20$  & $2.93\times10^{-3}$\tabularnewline
\hline 
 & $1\times10^{-6}$  & $3.94\times10^{7}$  & $30$  & $1.72\times10^{-3}$\tabularnewline
\hline 
\end{tabular}
\par\end{centering}

~

~

\protect\protect\protect\protect\caption{Table containing maximum condition numbers and final relative residual
errors of experiments with fixed separation ranks.\label{tab:Table-constant-rank}}
\end{table}

\section{Discussion and conclusions\label{sec:Discussion-and-conclusions}}

We have proposed a new ALS algorithm for reducing the rank of tensors
in canonical format that relies on projections onto random tensors.
Tensor rank reduction is one of the primary operations for approximations
with tensors. Additionally, we have presented a general framework
for the analysis of this new algorithm. The benefit of using such
random projections is the improved conditioning of matrices associated
with the least squares problem at each ALS iteration. While significant
reductions of condition numbers may be achieved, unlike in the standard
ALS, the application of random projections results in a loss of monotonic
error reduction. In order to restore monotonicity, we have employed
a simple rejection approach, wherein several random tensors are applied
and only those that do not increase the error are accepted. This,
however, comes at the expense of additional computational cost as
compared to the standard ALS algorithm. Finally, a set of numerical
experiments has been studied to illustrate the efficiency of the randomized
ALS in improving numerical properties of its standard counterpart.

The optimal choice of random variables to use in the context of projecting
onto random tensors is a question to be addressed in future work.
In our examples we have used signed Bernoulli random variables, a
choice that worked well with both our numerical experiments and analysis.
On the other hand, the limitations of such a construction of random
tensors have been discussed, which motivate further investigations.
Another topic of interest for future work is the extension of the
proposed randomized framework to other tensor formats including the
Tucker, \cite{KOL-BAD:2009}, and tensor-train, \cite{OSELED:2011}.

Finally we have suggested an alternative approach to using projections
onto random tensors that merits further examination. This approach
uses the QR factorization to construct a preconditioner for the least
squares problem at each ALS iteration. Hence it solves the same equations
as the standard ALS, but the matrices have better conditioning. Also,
because it solves the same equations, the monotonic error reduction
property is preserved. This is an important distinction from randomized
ALS, which solves different linear systems, but the solutions to which
are close to the solutions from standard ALS.

\section{Appendix~A\label{sec:Appendix-A}}

First, we prove (\ref{eq:rec-cond-num-ineq}). 
\begin{proof}
To bound the condition number of $AB$ we bound $\sigma_{\max}\left(AB\right)$
from above and $\sigma_{\min}\left(AB\right)$ from below. The bound
we use of $\sigma_{\max}\left(AB\right)$ is straightforward; it comes
from the properties of the two norm, 
\[
\sigma_{\max}\left(AB\right)\le\sigma_{\max}\left(A\right)\sigma_{\max}\left(B\right).
\]
To bound $\sigma_{\min}\left(AB\right)$ we first note that $AA^{T}$
is nonsingular, and write $\sigma_{\min}\left(AB\right)$ as follows,
\[
\sigma_{\min}\left(AB\right)=\left\Vert AA^{T}\left(AA^{T}\right)^{-1}AB\mathbf{x}^{*}\right\Vert _{2},
\]
where $\left\Vert \mathbf{x}^{*}\right\Vert =1$ is the value of $\mathbf{x}$
such that the minimum of the norm is obtained (see the minimax definition
of singular values, e.g. \cite[Theorem 3.1.2]{HOR-JOH:1994}). If
we define $\mathbf{y}=A^{T}\left(AA^{T}\right)^{-1}AB\mathbf{x}^{*}$,
then 
\[
\sigma_{\min}\left(AB\right)=\frac{\left\Vert A\mathbf{y}\right\Vert \cdot\left\Vert \mathbf{y}\right\Vert }{\left\Vert \mathbf{y}\right\Vert }\ge\sigma_{\min}\left(A\right)\cdot\left\Vert \mathbf{y}\right\Vert ,
\]
from the minimax definition of singular values. To bound $\left\Vert \mathbf{y}\right\Vert $,
we observe that $A^{T}\left(AA^{T}\right)^{-1}AB$ is the projection
of $B$ onto the row space of $A$. Denoting this projection as $P_{A^{T}}\left(B\right)$
we have, 
\[
\left\Vert \mathbf{y}\right\Vert =\left\Vert P_{A^{T}}\left(B\right)\mathbf{x}^{*}\right\Vert \ge\sigma_{\min}\left(P_{A^{T}}\left(B\right)\right),
\]
since $\left\Vert \mathbf{x}^{*}\right\Vert =1$. Combining our bounds
on the first and last singular values gives us the bound on the condition
number, 
\[
\kappa\left(AB\right)=\frac{\sigma_{\max}\left(A\right)\sigma_{\max}\left(B\right)}{\sigma_{\min}\left(A\right)\sigma_{\min}\left(P_{A^{T}}\left(B\right)\right)}=\kappa\left(A\right)\cdot\frac{\sigma_{\max}\left(B\right)}{\sigma_{\min}\left(P_{A^{T}}\left(B\right)\right)}.
\]

\end{proof}
The proof of Theorem~\ref{thm:Vershynin} is broken down into three
steps in order to control $\left\Vert A\mathbf{x}\right\Vert $ for
all $\mathbf{x}$ on the unit sphere: an approximation step, where
the unit sphere is covered using a finite epsilon-net $\mathcal{N}$
(see \cite[Section 5.2.2]{VERSHY:2012} for background on nets); a
concentration step, where tight bounds are applied to $\left\Vert A\mathbf{x}\right\Vert $
for every $\mathbf{x}\in\mathcal{N}$; and the final step where a
union bound is taken over all the vectors $\mathbf{x}\in\mathcal{N}$. 
\begin{proof}
(of Theorem~\ref{thm:Vershynin})

Vershynin observes that if we set $B$ in Lemma~\ref{lem:Vershynin}
to $A/\sqrt{N}$, the bounds on the extreme singular values $\sigma_{\min}\left(A\right)$
and $\sigma_{\max}\left(A\right)$ in (\ref{eq:Vershynin-theorem-bounds})
are equivalent to 
\begin{equation}
\left\Vert \frac{1}{N}A^{T}A-I\right\Vert <\max\left(\delta,\delta^{2}\right)=:\epsilon,\label{eq:vershynin-op-bound}
\end{equation}
where $\delta=C\sqrt{\frac{n}{N}}+\frac{t}{\sqrt{N}}$. In the approximation
step of the proof, he chooses a $\frac{1}{4}$-net $\mathcal{N}$
to cover the unit sphere $\mathcal{S}^{n-1}$. Evaluating the operator
norm (\ref{eq:vershynin-op-bound}) on $\mathcal{N}$, it is sufficient
to show 
\[
\max_{x\in\mathcal{N}}\left|\frac{1}{N}\left\Vert A\mathbf{x}\right\Vert _{2}^{2}-1\right|<\frac{\epsilon}{2},
\]
with the required probability to prove the theorem.

Starting the concentration step, \cite{VERSHY:2012} defines $Z_{i}=\left\langle A_{i},\mathbf{x}\right\rangle $,
where $A_{i}$ is the $i$-th row of $A$ and $\left\Vert \mathbf{x}\right\Vert _{2}=1$.
Hence, the vector norm may be written as 
\begin{equation}
\left\Vert A\mathbf{x}\right\Vert _{2}^{2}=\sum_{i=1}^{N}Z_{i}^{2}.\label{eq:vector-norm}
\end{equation}
Using an exponential deviation inequality to control (\ref{eq:vector-norm}),
and that $K\ge\frac{1}{\sqrt{2}}$, the following probabilistic bound
for a fixed $\mathbf{x}\in\mathcal{S}^{n-1}$ is, 
\begin{eqnarray}
\mathbb{P}\left\{ \left|\frac{1}{N}\left\Vert A\mathbf{x}\right\Vert _{2}^{2}-1\right|\ge\frac{\epsilon}{2}\right\}  & = & \mathbb{P}\left\{ \left|\frac{1}{N}\sum_{i=1}^{N}Z_{i}^{2}-1\right|\ge\frac{\epsilon}{2}\right\} \le2\,\exp\left[-\frac{c_{1}}{K^{4}}\min\left(\epsilon^{2},\epsilon\right)N\right]\nonumber \\
 & = & 2\,\exp\left[-\frac{c_{1}}{K^{4}}\delta^{2}N\right]\le2\,\exp\left[-\frac{c_{1}}{K^{4}}\left(C^{2}n+t^{2}\right)\right],\label{eq:vershynin-prob-norm-bound}
\end{eqnarray}
where $c_{1}$ is an absolute constant.

Finally, (\ref{eq:vershynin-prob-norm-bound}) is applied to every
vector $\mathbf{x}\in\mathcal{N}$ resulting in the union bound, 
\begin{equation}
\mathbb{P}\left\{ \max_{x\in\mathcal{N}}\left|\frac{1}{N}\left\Vert A\mathbf{x}\right\Vert _{2}^{2}-1\right|\ge\frac{\epsilon}{2}\right\} \le9^{n}\cdot2\,\exp\left[-\frac{c_{1}}{K^{4}}\left(C^{2}n+t^{2}\right)\right]\le2\,\exp\left(-\frac{c_{1}t^{2}}{K^{4}}\right),\label{eq:vershynin-union-bound}
\end{equation}
where we arrive at the second inequality by choosing a sufficiently
large $C=C_{K}$ (\cite{VERSHY:2012} gives the example $C=K^{2}\sqrt{\ln\left(9\right)/c_{1}}$). 
\end{proof}
We now prove Lemma~\ref{lem:modified-vershynin-lemma}. 
\begin{proof}
To prove this lemma we use the following inequality derived from (\ref{eq:vershynin-SV-bounds-B-sigma}),
\[
\left(1-\delta\right)^{2}\le\sigma_{\min}\left(\Sigma^{-\frac{1}{2}}B^{T}B\Sigma^{-\frac{1}{2}}\right)\le\sigma_{\max}\left(\Sigma^{-\frac{1}{2}}B^{T}B\Sigma^{-\frac{1}{2}}\right)\le\left(1+\delta\right)^{2}.
\]
First we bound $\sigma_{\max}\left(B\right)$ from above: 
\begin{eqnarray*}
\sigma_{\max}\left(B\right)^{2} & \le & \left\Vert \Sigma^{\frac{1}{2}}\right\Vert \cdot\left\Vert \Sigma^{-\frac{1}{2}}B^{T}B\Sigma^{-\frac{1}{2}}\right\Vert \cdot\left\Vert \Sigma^{\frac{1}{2}}\right\Vert \\
 & \le & \left\Vert \Sigma^{\frac{1}{2}}\right\Vert ^{2}\cdot\sigma_{\max}\left(\Sigma^{-\frac{1}{2}}B^{T}B\Sigma^{-\frac{1}{2}}\right)\\
 & \le & \sigma_{\max}\left(\Sigma^{\frac{1}{2}}\right)^{2}\cdot\left(1+\delta\right)^{2},
\end{eqnarray*}
implying $\sigma_{\max}\left(B\right)\le\sigma_{\max}\left(\Sigma^{\frac{1}{2}}\right)\cdot\left(1+\delta\right).$
Second we bound $\sigma_{\min}\left(B\right)$ from below: 
\begin{eqnarray}
\sigma_{\min}\left(B\right)^{2} & = & \sigma_{\min}\left(\Sigma^{\frac{1}{2}}\Sigma^{-\frac{1}{2}}B^{T}B\Sigma^{-\frac{1}{2}}\Sigma^{\frac{1}{2}}\right)\nonumber \\
 & \ge & \sigma_{\min}\left(\Sigma^{\frac{1}{2}}\right)^{2}\cdot\sigma_{\min}\left(\Sigma^{-\frac{1}{2}}B^{T}B\Sigma^{-\frac{1}{2}}\right)\nonumber \\
 & \ge & \sigma_{\min}\left(\Sigma^{\frac{1}{2}}\right)^{2}\cdot\left(1-\delta\right)^{2},\label{eq:smin-ineq-B}
\end{eqnarray}
implying $\sigma_{\min}\left(B\right)\ge\sigma_{\min}\left(\Sigma^{\frac{1}{2}}\right)\cdot\left(1-\delta\right).$
The first inequality in (\ref{eq:smin-ineq-B}) is from \cite[prob. 3.3.12]{HOR-JOH:1994}.
Finally, using properties of singular values we combine the inequalities:
\[
\sigma_{\min}\left(\Sigma^{\frac{1}{2}}\right)\cdot\left(1-\delta\right)\le\sigma_{\min}\left(B\right)\le\sigma_{\max}\left(B\right)\le\sigma_{\max}\left(\Sigma^{\frac{1}{2}}\right)\cdot\left(1+\delta\right).
\]

\end{proof}

\bibliographystyle{plain}
\bibliography{common}

\end{document}